\documentclass[11pt,reqno]{amsart}

\usepackage[dvipsnames]{xcolor}
\usepackage{tikz}
\usepackage{stmaryrd}
\usetikzlibrary{matrix,arrows,decorations.pathmorphing}
\usepackage{graphicx}

\usepackage{amssymb}
 \usepackage[stable]{footmisc}
\usepackage{amsmath,mathtools}
\usepackage{fullpage}
\usepackage{caption}
\usepackage{amsthm}
\usepackage{mathrsfs}
\usepackage{thmtools}
\usepackage{blkarray}
\usepackage{multirow}
\usepackage{picinpar} 
\usepackage{tikz-cd}
\usepackage{color}
\usepackage{verbatim}
\usepackage{hyperref}
\hypersetup{colorlinks=true, linkcolor=red, citecolor = OliveGreen, urlcolor = OliveGreen}
\usepackage{amssymb}
\usepackage{amsmath}
\usepackage{enumitem,colonequals}
\usepackage{color}
\usepackage{mathrsfs}
\usepackage[all]{xy}
\usepackage{comment}
\usepackage{mathpazo}
\usepackage{euler}
\usepackage[noabbrev,capitalise]{cleveref}
\usepackage{booktabs,array}
\usepackage[margin=1.in]{geometry}
\RequirePackage{mathrsfs} 

\newcommand{\mZ}{\mathbb{Z}}

\newcommand{\mQ}{\mathbb{Q}}

\newcommand{\mP}{\mathbb{P}}

\usepackage[OT2,T1]{fontenc}
\DeclareSymbolFont{cyrletters}{OT2}{wncyr}{m}{n}
\DeclareMathSymbol{\Sha}{\mathalpha}{cyrletters}{"58}
\DeclareMathSymbol{\Sha}{\mathalpha}{cyrletters}{"58}

\newcommand{\brk}[1]{ \left\lbrace #1 \right\rbrace }

\newcommand{\zZ}[1]{ \mathbb{Z}/#1\mathbb{Z}}

\newcommand{\iso}{\simeq}

\numberwithin{equation}{subsection}
\newtheorem{thmx}{Theorem}

\newtheorem{propx}[thmx]{Proposition}

\numberwithin{equation}{subsection}

\newtheorem{lemma}[subsection]{Lemma}

\newtheorem{conjecture}[subsection]{Conjecture}
\newtheorem{prop}[subsection]{Proposition}
\newtheorem{proposition}[subsection]{Proposition}

\theoremstyle{definition}
\newtheorem{defn}[subsection]{Definition}
\newtheorem{example}[subsection]{Example}

\theoremstyle{remark}

\newtheorem{remark}[subsection]{Remark}

\numberwithin{equation}{section} \numberwithin{figure}{section}

\DeclareMathOperator{\Gal}{Gal} 
\DeclareMathOperator{\Aut}{Aut} \DeclareMathOperator{\Spec}{Spec}

\DeclareMathOperator{\lcm}{lcm}

\DeclareMathOperator{\GL}{GL}

\DeclareMathOperator{\SL}{SL}

\newcommand{\Qbar}{\overline{\QQ}}

\newcommand{\cdef}[1]{\textsf{\textbf{#1}}}

\newcommand\ZZ{\mathbb{Z}}
\newcommand\Z{\mathbb{Z}}

\newcommand\QQ{\mathbb{Q}}

\renewcommand{\leq}{\leqslant}

\renewcommand{\geq}{\geqslant}

\newcommand{\Q}{\mathbb Q}
\newcommand{\arrow}{\longrightarrow}

\renewcommand{\Im}{{\rm Im}\,}

\makeatletter
\@namedef{subjclassname@1991}{\emph{2020} Mathematics Subject Classification}
\makeatother

\begin{document}
\title{A group theoretic perspective on entanglements of division fields}

\author{Harris B. Daniels} 
\address{Harris B. Daniels \\
	Department of Mathematics \\ 
	Amherst College \\
	Amherst, MA 01002
    USA}
\email{hdaniels@amherst.edu}

\author{Jackson S. Morrow} 
\address{Jackson S. Morrow \\
	Centre de Recherche de Math\'ematiques, Universit\'e de Montr\'eal\\
    Montr\'eal, Qu\'ebec H3T 1J4\\
    CAN}
\email{jmorrow4692@gmail.com}


\date{\today}

\begin{abstract}
In this paper, we initiate a systematic study of entanglements of division fields from a group theoretic perspective. 
For a positive integer $n$ and a subgroup $G\subseteq \GL_2(\zZ{n})$ with surjective determinant, we provide a definition for $G$ to represent an $(a,b)$-entanglement and give additional criteria for $G$ to represent an explained or unexplained $(a,b)$-entanglement.

Using these new definitions, we determine the tuples $((p,q),T)$, with $p<q\in\ZZ$ distinct primes and $T$ a finite group, such that there are infinitely many non-$\overline{\mQ}$-isomorphic elliptic curves over $\QQ$ with an unexplained $(p,q)$-entanglement of type $T$. 
Furthermore, for each possible combination of entanglement level $(p,q)$ and type $T$, we completely classify the elliptic curves defined over $\QQ$ with that combination by constructing the corresponding modular curve and $j$-map.
\end{abstract}
\maketitle

\section{\bf Introduction}
\label{sec:Intro}

Given an elliptic curve $E/\QQ$ and fixing an algebraic closure of $\QQ$, call it $\Qbar$, one can construct a Galois representation associated to $E$
\[
\rho_E\colon \Gal(\Qbar/\QQ)\to\GL_2(\widehat{\ZZ})\simeq \prod_{\ell{ \rm\ prime}}\GL_2(\ZZ_{\ell}).
\]
The image of this representation (which is only defined up to conjugation) encodes information about the the fields of definition of the torsion points on $E$. 
In \cite{serre:OpenImageThm}, Serre showed that if $E$ is non-CM, then $d_E := [\GL_2(\widehat{\ZZ}) :\Im(\rho_E)]\geq 2$, and work of Duke \cite{Duke} and Jones \cite{JonesSerreCurves} proved that for almost all elliptic curves $E/\QQ$ (in the sense of density), $d_E=2$. 
In light of these results, it is natural to wonder:~what conditions make the index $d_E$ greater than 2?

Using the above isomorphism, we can see that there are two (seemingly) orthogonal ways the image of $\rho_E$ could be smaller than expected:
\begin{enumerate}
\item The image of $\rho_E$ composed with projection onto $\GL_2(\ZZ_{\ell})$ might be smaller than expected (i.e., the $\ell$-adic Galois representation associated to $E$ might not be surjective).
\item Relatively prime division fields with non-trivial intersection force the image of $\rho_E$ to not be the {\it full} direct product of the images of the $\ell$-adic representations.
\end{enumerate}
Recently, there has been great progress in understanding and classifying the ways in which an elliptic curve over $\QQ$ can fail to have surjective $\ell$-adic Galois representation; see for example \cite{rouse2014elliptic, sutherlandzywina2017primepower, RouseSZB:elladic}. 
On the other hand, there has been recent progress \cite{morrow2017composite, DanielsGJ, morrowLowComposite} in determining what composite images can occur. 
However, there has yet to be a systematic study of the way in which the division fields of elliptic curves can intersect.

\subsection*{Main contributions}
The goal of this article is to lay the foundation for the study of entanglements of division fields of elliptic curves using group theoretic methods. 
To date, there has been limited study of the way in which the division fields can have non-trivial intersection; see for example \cite{brau3in2, morrow2017composite, danielsLR:coincidences, JonesM:Nonabelianentanglements}. 
To initiate our discussion, we give a formal definition, which captures when division fields have non-trivial and unexpected intersection.

\begin{defn}\label{defn:firstentanglement}
Let $E/\QQ$ be an elliptic curve and let $a,b$ be positive integers. We say that $E$ has an \cdef{$(a,b)$-entanglement} if 
\[
K = \QQ(E[a]) \cap \QQ(E[b])\neq \QQ(E[d]),
\] 
where $d=\gcd(a,b)$. 
The \cdef{type} $T$ of the entanglement is the isomorphism class of $\Gal(K/\QQ(E[d]))$.
\end{defn}

These entanglements come in two flavors. 
To see this distinction, we explain the observation that Serre made in \cite{serre:OpenImageThm} to prove that $d_E\geq 2$ for all $E/\QQ$. 
Fix an arbitrary elliptic curve $E/\QQ$ such that its minimal discriminant $\Delta_E$ is not a square. 
It is a classical result that $\Q(\sqrt{\Delta_E})\subseteq \Q(E[2])$ (see \cite{Adelmann:Decomp} for example). Further, from the Kronecker--Weber theorem, there is some $m>2$ such that $\Q(\sqrt{\Delta_E})\subseteq \Q(\zeta_m)$, so that $\Q(E[2])\cap \Q(\zeta_m)$ is a non-trivial quadratic extension of $\Q$, and one of the {many} consequences of the Weil pairing is that $\QQ(\zeta_m)\subseteq\QQ(E[m])$ ensuring that, in this case $\Q(E[2])\cap \Q(E[m])$ is non-trivial. 
This nontrivial intersection of division fields is what causes the index $d_E$ to be at least $2$.
We note that if the minimum $m$ is odd, then we have a $(2,m)$-entanglement of type $\zZ{2}$, and when $m$ is even, one finds a $(4,m/4)$-entanglement of type $\zZ{2}$ coming from $\mQ(\sqrt{-\Delta_E}) \subset \mQ(E[4]) \cap \mQ(E[m/4])$ or a $(8,m/8)$-entanglement of type $\zZ{2}$ coming from $\mQ(\sqrt{\pm 2\Delta_E}) \subset \mQ(E[8]) \cap \mQ(E[m/8])$. 
We say that such an entanglement is \textit{explained} because its existence can be explained through the Weil pairing (i.e., that $\QQ(\zeta_n) \subseteq \QQ(E[n])$ for all $E/\QQ$) and the Kronecker--Weber theorem (every abelian extension of $\QQ$ is in $\QQ(\zeta_n)$ for some $n$).

In this work, we will primarily be concerned with \textit{unexplained} entanglements. Below, we provide three examples of this flavor of entanglement. 

\begin{example}\label{exam:firstunexplainedentanglement}
Let $E/\mQ$ be the elliptic curve with Cremona label \href{https://www.lmfdb.org/EllipticCurve/Q/100a3/}{\texttt{100a3}} which has Weierstrass model
\[
E\colon y^2 = x^3 - x^2 - 1033x - 12438.
\]
We claim that $E$ has an entanglement between its 2- and 3-division fields that is not explained by the Weil pairing and the Kronecker--Weber theorem, in particular we have that $\mQ(E[2]) \cap \mQ(E[3]) = \mQ(\sqrt{5})$. 
Indeed, this elliptic curve has a rational 3-isogeny $\varphi\colon E \to E'$, and we compute that the kernel of this isogeny is generated by a point $P$ of order 3 with field of definition $\mQ(P) = \mQ(\sqrt{-15})$. 
We also determine that the 2-division field $\mQ(E[2])$ is isomorphic to $\mQ(\sqrt{5})$, and since $\mQ(\sqrt{-3}) $ is contained in the 3-division field $ \mQ(E[3])$ by the Weil pairing, we know that $\mQ(E[2]) \cap \mQ(E[3]) \supseteq \mQ(\sqrt{5})$ and in fact  $\mQ(E[2]) \cap \mQ(E[3]) = \mQ(\sqrt{5})$. Clearly, this entanglement is not explained by the Weil pairing and the Kronecker--Weber theorem since the Weil pairing only ensures that $\sqrt{-3}\in\QQ(E[3])$ and does not tell us anything about $\sqrt{5}$ being in $\QQ(E[2])$ or $\Q(E[3])$. 
\end{example}

\begin{example}\label{exam:dual}
Let $E/\mQ$ be the elliptic curve with Cremona label \href{https://www.lmfdb.org/EllipticCurve/Q/4225f2/}{\texttt{4225f2}} which has Weierstrass model
\[
E\colon y^2 + y = x^3 - x^2 - 2708x + 54693.
\]
We start with a few observations about $E$. First, the discriminant $\Delta_E$ of $E$ is equivalent to 13 modulo rational squares, and thus $\QQ(\sqrt{13})\subseteq\QQ(E[2])$. Secondly, $E$ has a rational $5$-isogeny, call it $\varphi$. Letting $P$ be a generator of the cyclic kernel of $\varphi$, we see using \texttt{Magma} that $\QQ(x(P)) = \QQ(\sqrt{65})\subseteq\QQ(E[5])$. From the Weil-pairing, we know that $\QQ(\sqrt{5})\subseteq\QQ(\zeta_5)\subseteq\QQ(E[5])$ and so again we have a quadratic intersection between two relatively prime division fields.

Let $E'/\QQ$ be the codomain of $\varphi$, let $\widehat{\varphi}\colon E'\to E$ be the dual isogeny of $\varphi$, and let $Q$ be a generator for the kernel of $\widehat{\varphi}$. Using \texttt{Magma}, we again see that $\Delta_{E'} \equiv 13 \bmod (\QQ^\times)^2$ and that $\QQ(x(Q)) = \QQ(\sqrt{13})$, and hence $E'$ has a quadratic entanglement between its 2- and 5-division fields; however, this time the entanglement occurs inside the kernel of $\widehat{\varphi}$. 
The point is that since $\QQ(E[5]) = \QQ(E'[5])$, it is sensible to say that $E$ has a $(2,5)$-entanglement contained inside the kernel of the dual of its isogeny; in particular, the isogenous curve $E'$ has a $(2,5)$-entanglement contained in the kernel of its isogeny. Again, these entanglements cannot be explained by the Weil pairing and/or the Kronecker--Weber theorem.
\end{example}

\begin{example}
Let $E/\mQ$ be the elliptic curve with Cremona label \href{https://www.lmfdb.org/EllipticCurve/Q/5780c1/}{\texttt{5780c1}} which has Weierstrass model
\[
E\colon y^2 = x^3 - 272x - 1564.
\]
This elliptic curve has $\Im(\rho_{E,5})$ contained in the exceptional group $G_{S_4}(5)$, which is a maximal subgroup of $\GL_2(\ZZ/5\ZZ)$. The 5-division polynomial $f_5(x)$ of $E$ is an irreducible polynomial of degree 12, and if we let $L$ be the splitting field of $f_5(x)$, then there is a unique degree 6 Galois extension $K/\QQ$ such that $K\subseteq L$. A quick check in \texttt{Magma} shows that $E[2]\subseteq E(K)$ and since 2 is not exceptional for $E$, we know that $K = \QQ(E[2])$. Thus, there is a non-trivial intersection between the 2- and 5-division  fields of $E$ and if we let $F = \QQ(E[2])\cap\QQ(E[5])$, then $\Gal(F/\QQ)\simeq S_3$.
Since $F$ is not an abelian extension, it is not possible for the entanglement to be \textit{completely} explained by the Weil pairing and the Kronecker--Weber theorem, but looking closer we see that the unique quadratic extension of $\QQ$ contained in $F$ is actually $\QQ(\sqrt{\Delta_E}) \simeq \QQ(\sqrt{5})$. Therefore, this part of the entanglement is explained by the Weil pairing and/or the Kronecker--Weber theorem, but the rest of the entanglement remains unexplained.
\end{example}

\subsection*{Statement of results}
As mentioned above, the goal of our work is to establish a group theoretical foundation to studying entanglements. 
The idea behind this approach is to translate the property of an elliptic curve $E/\QQ$ having an $(a,b)$-entanglement into group theoretic conditions on the image of the mod $n$ Galois representation associated to $E$ where $n=\lcm(a,b)$.
Once this is done, for relevant $n$, we can search for certain subgroups of $\GL_2(\ZZ/n\ZZ)$ that \emph{represent} entanglements (see Definition \ref{def:representsentanglement}) and then study the associated modular curve. 
This process allows us to reduce the question of classifying elliptic curves with non-trivial entanglements to a group theoretic problem and a question of rational points on curves. 

As an application of the framework we lay out, we prove the following theorem.

\begin{thmx}\label{thm:main}
There are exactly 9 pairs $((p,q),T)$, with $p<q\in\ZZ$ distinct primes and $T$ a finite group such that there are infinitely many non-$\overline{\mQ}$-isomorphic elliptic curves over $\QQ$ with an unexplained $(p,q)$-entanglement of type $T$ (Definitions \ref{def:unexplained} and \ref{def:elliptic_unexplained}). 

For each possible combination of entanglement level $(p,q)$ and type $T$, we completely classify the elliptic curves defined over $\QQ$ with that combination. 
\end{thmx} 

In order to classify the elliptic curves defined over $\QQ$ with each combination, we need to compute 24 different modular curves:~22 of which have genus 0 and 2 of which are genus 1 with positive rank. The details of the classification can be found in Section \ref{sec:tables}.

During our analysis of the genus 1 groups, we encounter an exceptional isomorphism. 
The two genus 1 groups we find both represent a $(2,7)$-entanglement of type $\ZZ/2\ZZ$. They both have full image mod 2, and they only differ in their mod 7 image:~one of the groups has mod 7 image conjugate to the normalizer of the non-split Cartan at level 7 and the other has mod 7 image conjugate to the normalizer of the split Cartan at level 7. We denote these level 14 groups by $G_n$ and $G_s$, respectively; the generators for these groups can be found in Section \ref{sec:genusone}.  
The amazing thing is that the modular curves associated to these two groups are actually isomorphic over $\QQ$.

\begin{propx}\label{prop:iso}
The genus one modular curves of positive rank $X_{G_s}$ and $X_{G_n}$ are both isomorphic over $\mQ$ to the elliptic curve \href{https://www.lmfdb.org/EllipticCurve/Q/196a1/}{\emph{\texttt{196a1}}}. 
\end{propx}

A priori, there is no reason to expect that these curves would be isomorphic over $\Qbar$ let alone over $\QQ$. 
Looking at the computations in Section \ref{sec:genusone}, it is unclear if this is a coincidence or if there is something deeper happening (cf.~\cite{baran:exceptionalisomorphism}). 
Filip Najman pointed out that there is no ``moduli interpretation'' of the isomorphism as the curves have a different number of cusps over $\mQ$. 
While we cannot explain this isomorphism, one of the referees alerted us to the fact that we can precisely describe the entanglement field of the moduli of $X_{G_n}$ and $X_{G_s}$. 
If $E/\mQ$ denotes an elliptic curve whose mod 14 image is conjugate to $G_n$ (resp.~$G_s$), then we have that quadratic subfield of $\mQ(E[7])$ that coincides with $\mQ(\sqrt{\Delta_E})$ is the fixed field of $\mQ(E[7])$ by the non-split Cartan subgroup at level 7 (resp.~the split Cartan subgroup at level 7). In the case where $E/\mQ$ has CM, this fixed field agrees with the CM field of $E$ (see the table in Subsection \ref{subsection:genus1Table} for examples).

\subsection*{Overview of proof of Theorem \ref{thm:main}}
Our proof of Theorem \ref{thm:main} is inspired by the proofs of the main theorems of \cite{rouse2014elliptic, zywinapossible, sutherlandzywina2017primepower}. 
In these works, the authors begin by defining a purely group theoretic notion which captures the moduli problem of interest and computing all of the subgroups of a certain general linear group which satisfy their notion. 
Once they have the complete list of such groups, they determine the equation of the corresponding modular curve, and then proceed with a determination of the rational points on these modular curves.

In more detail, we break the proof of Theorem \ref{thm:main} into a few steps and point the reader to where they can find more information about each step within the paper. 

\medskip

\noindent\underline{\sc Step 0}: 
Let $\mathscr{L}$ be the set of natural numbers $n$ that are products of two distinct primes such that there exists a congruence subgroup $\Gamma$ of level $n$ where the genus of $X_\Gamma$ is either 0 or 1. 
From \cite{CoxParry}, we know that $\mathscr{L}$ is a finite list, and the work of  \cite{CumminsPauli} makes the list explicit:
\[
\mathscr{L} = \{ 6, 10, 14, 15, 21, 22, 26, 33, 34, 38, 39 \}.
\]
Using \cite[Corollary 1.6]{sutherlandzywina2017primepower}, we may immediately exclude the levels $34$ and $38$ as there are no modular curves of levels $17$ or $19$, which have infinitely many $\mQ$-rational points. 
Therefore, we can refine the list $\mathscr{L}$ to 
\[
\mathscr{L}' = \{ 6, 10, 14, 15, 21, 22, 26, 33, 39 \}.
\]
For each $n\in \mathscr{L}'$, we compute the sets 
\[
\mathcal{G}_0(n) =\{G \subseteq \GL_2(\ZZ/n\ZZ) : G\hbox{ is admissible and the modular curve $X_G$ has genus 0}\}, 
\]
and 
\[
\mathcal{G}_1(n) =\{G \subseteq \GL_2(\ZZ/n\ZZ) : G\hbox{ is admissible and the modular curve $X_G$ has genus 1}\},
\]
where the we refer the reader to Definition \ref{3.1} for the definition of admissible. 
We will search for groups that represent unexplained entanglements from these sets. For the definition of representing an unexplained entanglement, see Definitions \ref{def:representsentanglement}, \ref{defn:primitiveentanglement}, and \ref{def:unexplained}. 

We pause here to remark that it suffices to check only these levels. If $G \subseteq \GL_2(\zZ{n})$ represents a primitive $(p,q)$-entanglement (Definition \ref{defn:primitiveentanglement}), then the corresponding $\Gamma$ either has level $pq$, level $p$, level $q$, or is all of $\SL_2(\mZ)$. If $\Gamma$ has level $p$, level $q$, or is all of $\SL_2(\mZ)$, then $\Gamma$ represents a trivial entanglement, and hence $G$ represents an explained entanglement (Definition \ref{def:explained}) and thus can be excluded. 
\medskip

\noindent
\underline{\sc Step 1a}: 
For each $n\in\mathscr{L}'$, we determine which groups $G\in \mathcal{G}_0(n)$ represent an unexplained entanglement. 
For these groups, we sort them into sets $\mathcal{G}_{0,k}(n)$ based on their entanglement level and type, where $k$ is an ordered pair of the form $((a,b),T)$. 
\medskip

\noindent
\underline{\sc Step 1b}: 
For each $n\in\mathscr{L}'$, we first determine which groups $G\in \mathcal{G}_1(n)$ have the property that their corresponding genus 1 modular curve $X_G$ has positive rank over $\mQ$. 
Using \cite{sutherlandzywina2017primepower}, this can be determined using only the group $G$ and without computing $X_G$. 
Once we have eliminated the groups that correspond to rank zero elliptic curves, we proceed just as in {\sc Step 1a}. Again we sort the groups into sets, only this time $\mathcal{G}^+_{1,k}(n)$ has the added condition that $X_G$ is not only genus 1, but also that it has positive rank over $\QQ$.
\medskip

\noindent
\underline{\sc Step 2}: 
For each set of the form $\mathcal{G}_{0,k}(n)$ and $\mathcal{G}_{1,k}^+(n)$, we compute the set of maximal elements (with containment, up to conjugation) and put those into sets of the form $\mathcal{M}_{0,k}(n)$ or $\mathcal{M}_{1,k}(n)$. The values for which these sets are non-empty correspond to the 9 pairs $((a,b),T)$ in the statement of Theorem \ref{thm:main}. 
For a concrete example of why we take this step, see Example \ref{Ex:MaxGroups}.
\medskip

\noindent
\underline{\sc Step 3}: 
The last step is to determine a model for the modular curve $X_G$ over $\mQ$ as well as its $j$-map for each $G\in \mathcal{M}_{0,k}$ or $\mathcal{M}_{1,k}(n)$. 
The details of this step can be found in Sections \ref{sec:computataions}, \ref{sec:UnexplainedGenus0}, and \ref{sec:UnexplainedGenus1}.

\subsection*{Related results}
Brau--Jones \cite{brau3in2} and the second author \cite[Theorem 8.7]{morrow2017composite} have classified all elliptic curves $E/\Q$ with $(2,3)$-entanglement of non-abelian type. 
In recent work, Lozano-Robledo and the first author \cite{danielsLR:coincidences} classified the elliptic curves $E/\Q$, and primes $p$ and $q$ such that $\mQ(E[p])\cap \mQ(\zeta_{q^k})$ is non-trivial and determined the degree of this intersection. As a consequence, they also classify all elliptic curves $E/\Q$ and integers $m,n$ such that the $m$-th and $n$-th division fields coincide.
Recently, Campagna--Pengo \cite{CampagnaAndPengo} have studied the entanglements of CM elliptic curves focusing on when division fields become linearly disjoint, and they used their results to determine the index of the adelic image of Galois associated to a CM elliptic curve over $\mQ$ inside of the normalizer of a certain Cartan subgroup (see \textit{loc.~cit.~}Corollary 4.6 and Remark 4.7 for details). 
Finally, Jones--McMurdy \cite{JonesM:Nonabelianentanglements} determine the genus zero modular curves and their $j$-maps whose rational points correspond to elliptic curves with entanglements of non-abelian type. 

\subsection*{Comments on code}
We mention here that the code and techniques used throughout this paper build upon those of previous results. The authors would like to especially point out the results and code  in the following articles \cite{zywinapossible, rouse2014elliptic, Sutherland, sutherlandzywina2017primepower, DanLRNSQ3tor}. We also relied heavily on the LMFDB \cite{lmfdb}, in order to understand and generate the examples in this paper. 

All of the computations in this paper were performed using \texttt{Magma} \cite{Magma}. The code used to do the computations can be found at the following link. 
\begin{center}
{\color{black}{\url{https://github.com/jmorrow4692/Entanglements}}}
\end{center}

\subsection*{Outline of paper}
In Section \ref{sec:preliminaries}, we recall some background on Galois representations attached to elliptic curves, modular curves, and Siegel functions. 
In Section \ref{sec:grouptheoreticdefinitions}, we give our group theoretic definition of entanglements and an additional criteria for G to represent an explained or unexplained entanglement. 
In Section \ref{sec:entanglementsellipticcurves}, we show how the group theoretic definitions codify entanglements of division fields of elliptic curves and provide several examples to illustrate the relationship between the notions.

Section \ref{sec:computationsmodularcurves} begins the description of the modular curve computations and gives explicit examples of these computations for a genus 0 subgroup with $-I$ using Siegel functions and for genus 0 subgroups without $-I$. 
In Sections \ref{sec:UnexplainedGenus0} and \ref{sec:UnexplainedGenus1}, we discuss computations of modular curves of genus 0 and 1, respectively, which were done using different methods. 
Finally, in Section \ref{sec:tables}, we provide tables of the various modular curves whose rational points parametrize elliptic curves with an unexplained $(p,q)$-entanglement of type $T$ for the pairs $((p,q),T)$ mentioned in Theorem \ref{thm:main}. 

\subsection*{Acknowledgements} 
The authors would like to thank Jeffrey Hatley, Nathan Jones, \'Alvaro Lozano--Robledo, Ken McMurdy, and David Zureick-Brown for helpful conversations. 
We are especially indebted to Jeremy Rouse for his help with Example \ref{exam:Jeremy} and his support on this project. 
Also the authors would like to thank Garen Chiloyan, Enrique Gonz\'alez-Jim\'enez, Jeffrey Hatley, \'Alvaro Lozano--Robledo, and Filip Najman for useful comments on an earlier draft and to extend their thanks to Lea Beneish for detailed and thoughtful comments on an earlier draft. 
The authors would like to especially thank Nathan Jones for his thorough reading of the manuscript and for several corrections. 
Finally, we would like to thank the two anonymous referees for their detailed and thoughtful comments.

\subsection*{Conventions}
Throughout, we will use the following conventions.

\subsubsection*{Groups}
We set some notation for specific subgroups of $\GL_2(\zZ{\ell})$. Let $\ell{\texttt Cs}$ be the subgroup of diagonal matrices. 
Let $\epsilon = -1$ if $\ell \equiv 3 \pmod 4$ and otherwise let $\epsilon \geq 2$ be the 
smallest integer which is not a quadratic residue modulo $\ell$. 
Let $\ell{\texttt Cn}$ be the subgroup consisting of matrices of the form $\begin{psmallmatrix} a & b\epsilon \\ b & a \end{psmallmatrix}$ with $(a,b) \in \zZ{\ell}^2 \setminus \brk{(0,0)}$. 
Let $\ell{\texttt Ns}$ and $\ell{\texttt{Nn}}(\ell)$ be the normalizers of $\ell{\texttt Cs}$ and $\ell{\texttt Cn}$, respectively, in $\GL_2(\zZ{\ell})$. 
We have $[\ell{\texttt{Ns}} : \ell{\texttt{Cs}}] = 2$ and the non-identity coset of $\ell{\texttt{Cs}}(\ell)$ in $\ell{\texttt{Ns}}(\ell)$ is represented by $\begin{psmallmatrix} 0 & 1 \\ 1 & 0 \end{psmallmatrix}$. 
Similarly, $[\ell{\texttt{Nn}} : \ell{\texttt{Cn}}] = 2$ and the non-identity coset of $\ell{\texttt{Cn}}$ in $\ell{\texttt{Nn}}$ is represented by $\begin{psmallmatrix} 1 & 0 \\ 0 & -1 \end{psmallmatrix}$. Let $\ell{\texttt B}$ be the subgroup of upper triangular matrices in $\GL_2(\zZ{\ell})$. 
This notation was established by Sutherland in \cite{Sutherland} and is used in the LMFDB \cite{lmfdb}. We will also use Sutherland's notation for the less standard subgroups of level $p$. 

When studying an entanglement group $G$ of composite level, it is useful to keep track of their images mod $a$ and $b$. To this end, if we have a group of level $pq$, we will assign it a label of $[L_p(G),L_q(G)]$ where $L_p(G)$ (resp.~$L_q(G)$) is the Sutherland \cite{Sutherland} label of the image of $G$ mod $p$ (resp.~mod $q$).

Finally, we will use the notation $\pi_a$ to denote the reduction map of $\GL_2(\zZ{n}) \to \GL_2(\zZ{a})$ where $a\mid n$ and $n$ should be clear from context.

\subsubsection*{Elliptic curves}
For a field $k$, we will use $E/k$ to denote an elliptic curve over $k$. 
For a square-free  element $d \in k^{\times}/(k^{\times})^2$, the twist of $E$ by $d$ will be denoted by $E^{(d)}$. 
Any particular elliptic curve over $\mQ$ mentioned in the paper will be given by Cremona reference and a link to the corresponding LMFDB \cite{lmfdb} page when possible. 

\section{\bf Preliminaries}
\label{sec:preliminaries}
In this preliminary section, we recall background on Galois representations associated to elliptic curves, modular curves, and Siegel functions. 

\subsection{Galois representations of elliptic curves and modular curves}
Let $E$ be an elliptic curve over $\mQ$. 
For any positive integer $n$, we denote the $n$-torsion subgroup of $E(\overline{\mQ})$, where $\overline{\mQ}$ is a fixed algebraic closure of $\mQ$, by $E[n]$. For a prime $\ell$, let 
\begin{equation*}
T_{\ell}(E)\coloneqq \varprojlim_{n\geq 1}E[\ell^n]
\hbox{\ \ \ and\ \ \ }
T(E) \coloneqq \varprojlim_{n\geq 1} E[n]
\end{equation*}
denote the $\ell$-adic Tate module and adelic Tate module, respectively. 
By fixing a $\widehat{\mZ}$-basis for $T(E)$, there is an induced $\zZ{n}$-basis on $E[n]$ for any positive integer $n$. The absolute Galois group $G_{\mQ} \coloneqq \Gal (\overline{\mQ}/\mQ)$ has a natural action on each torsion subgroup, which respects each group structure. In particular, we have the following continuous representations
\begin{align*}
\rho_{E,n}\colon &G_{\mQ}  \longrightarrow  \Aut(\zZ{n}) \iso \GL_2(\zZ{n})  &(\text{mod }n)&,\\
\rho_{E,\ell^{\infty}} \colon & G_{\mQ} \longrightarrow  \Aut(T_{\ell}(E))  \iso \GL_2(\mZ_{\ell}) &(\ell\text{-adic}),&\\
\rho_{E}\colon & G_{\mQ} \longrightarrow \Aut(T(E)) \iso \GL_2(\widehat{\mZ}) &(\text{adelic}),& 
\end{align*}
where the image under $\rho$ is uniquely determined up to conjugacy in its respective general linear group. 
The \cdef{$n$-division field} $\mQ(E[n])$ is the fixed field of $\overline{\mQ}$ by the kernel of the mod $n$ representation;  moreover, the Galois group of this number field is the image of the mod $n$ representation.

A celebrated theorem of Serre \cite{serre:OpenImageThm} says that for a non-CM elliptic curve $E/\mQ$, the adelic representation $\rho_E$ has open image in $\GL_2(\widehat{\mZ})$. 
Using the isomorphism 
\[
\GL_2(\widehat{\mZ})  \simeq \prod_{\ell \text{ prime}} \GL_2(\mZ_{\ell}), 
\]
we see that for any non-CM elliptic curve over $\mQ$, there exists a smallest integer $r_{E/\mQ} >0$ such that for all $\ell\geq r_{E/\mQ}$ , the $\ell$-adic representation is surjective. Serre \cite[p.~399]{Serre:applicationsChebotarev} asked whether $r_{E/\mQ} = 41$. 
In \cite[Conjecture 1.1]{zywina2011surjectivity}, Zwyina gave a refined conjecture concerning the surjectivity of the mod $\ell$ image and provided a practical algorithm (implemented in \texttt{Sage}) to compute the finite set of primes $\ell$ for which $\rho_{E,\ell}(G_{\mQ})$ is not surjective; a prime $\ell$ is called \cdef{exceptional} if it belongs to this finite set. 
Finally, Sutherland \cite{Sutherland} performed extensive computations on determining the mod $\ell$ image of Galois for elliptic curves in the Cremona tables \cite{lmfdb} and in the Stein--Watkins database \cite{SteinWatkins}, which led to further refinements of these conjectures (see \cite[Conjecture 1.1]{Sutherland}). 
Below, we provide a version of their conjectures.

\begin{conjecture}[Serre, Sutherland, Zywina]
Let $E/\mQ$ be a non-CM elliptic curve. If $\ell > 37$, then $\rho_{E,\ell}$ is surjective. 
\end{conjecture}

We now describe a set of necessary conditions on the possible non-surjective images of $\rho_{E,N}(G_{\mQ})$, where $N \geq 2$. We follow closely the conventions laid out in \cite{sutherlandzywina2017primepower} modifying only the condition that $-I\in G$. 
\begin{defn}\label{3.1}
A subgroup $G$ of $\GL_2(\zZ{N})$ is \cdef{admissible} if it satisfies the following conditions:
\begin{itemize}
\item $G \neq \GL_2(\zZ{N})$,
\item $\det (G) = (\zZ{N})^\times$,
\item $G$ contains an element with trace 0 and determinant $-1$ that fixes a point in $(\zZ{N})^2$ of order $N$.
\end{itemize}
\end{defn}

\begin{prop}\protect{\cite[Proposition~2.2]{zywinapossible}} \label{3.2}
Let $E$ be an elliptic curve over $\mQ$ for which $\rho_{E,N}(G_{\mQ})$ is not surjective. Then $\rho_{E,N}(G_{\mQ})$ is an admissible subgroup of $\GL_2(\zZ{N})$.
\end{prop}

For an admissible subgroup $G \subseteq \GL_2(\zZ{N})$ with $-I\in G$, we can associate to it a modular curve $X_G$, which is a smooth, projective, and geometrically irreducible curve over $\mQ$. It comes with a natural morphism 
\begin{equation*}
\pi_G\colon X_G \longrightarrow \Spec \mQ[j] \cup \brk{\infty} \eqqcolon \mP^1_{\mQ},
\end{equation*}
such that for an elliptic curve $E/\mQ$ with $j_E \notin \brk{0,1728}$, the group $\rho_{E,n}(G_{\mQ})$ is conjugate to a subgroup of $G$ if and only if $j_E = \pi_G(P)$ for some rational point $P \in X_G(\mQ)$.

There has been extensive work on determining the modular curves $X_G$ and their associated $j$-maps $\pi_G$. 
More precisely, Zywina \cite{zywinapossible} has classified $(X_G,\pi_G)$ where $G \subseteq\GL_2(\zZ{\ell})$ is an admissible subgroup. 
Rouse--Zureick-Brown \cite{rouse2014elliptic} have determined $(X_G,\pi_G)$ where $G \subseteq\GL_2(\zZ{2^n})$ is an admissible subgroup, and Sutherland--Zywina \cite{sutherlandzywina2017primepower} have computed $(X_G,\pi_G)$ where $G \subseteq\GL_2(\zZ{\ell^n})$ is an admissible subgroup and the associated modular curve $X_G$ has genus 0 or genus 1 and positive rank. 
Our modular curve computations use many of the techniques laid out in these works, especially \cite{sutherlandzywina2017primepower}.

When the admissible subgroup $G$ does not contain $-I$, one cannot just work with the coarse space to understand the moduli interpretation because further information is required. See \cite[Sections 2 and 5]{rouse2014elliptic}.

Before describing a technique to compute these modular curves, we prove a lemma relating the mod $n$ image of Galois for $n$-isogenous elliptic curves $E,E'$ over $\mQ$. 

\begin{lemma}\label{lem:isogrep}
Let $E_1,E_2$ be elliptic curves over $\mQ$.
Let $\phi\colon E_1 \to E_2$ be a cyclic $n$-isogeny defined over $\mQ$ with kernel $\langle P_1 \rangle$, and let $\widehat{\phi}\colon E_2 \to E_1$ denote the dual isogeny. 
Fix a basis $\brk{P_1,Q_1}$ for the $n$-torsion on $E_1$. 
Let $P_2 = \phi(Q_1)$ and fix a basis $\brk{P_2,Q_2}$ for the $n$-torsion on $E_2$. 

Let $\sigma$ be an element of $\Gal(\Qbar/\QQ)$. 
If
\[
\rho_{E_1,n}(\sigma) = \begin{pmatrix} a& b\\ 0& d \end{pmatrix},
\]
where $a,d \in (\mZ/n\mZ)^\times$ and $b\in (\mZ/n\mZ)$, then there exists a $\beta\in(\ZZ/n\ZZ)$ such that 
\[
\rho_{E_2,n}(\sigma) = \begin{pmatrix} d& \beta\\ 0& a \end{pmatrix}.
\]
\end{lemma}

\begin{proof}
Our assumption on the mod $n$ representation of $E_1$ tells us that for an element $\sigma \in G_{\mQ}$, the action of $\sigma$ on $E_1[p]$ can be described as:
\begin{align*}
\sigma(P_1) &= aP_1 \\
\sigma(Q_1) &= bP_1 + dQ_1.
\end{align*}
Since $\phi$ is defined over $\mQ$ (in particular, $\phi$ is Galois invariant) and a group homomorphism, we have that 
\begin{align*}
\sigma(P_2) &= \sigma(\phi(Q_1)) 
= \phi(\sigma(Q_1)) 
= \phi(bP_1 + dQ_1) 
= b\phi(P_1) + d\phi(Q_1) = dP_2. 
\end{align*}
Using properties of the Weil pairing (see for example \cite[Proposition III.8.1]{Silverman}), we know that $\det\circ\rho_{E,n}\colon\Gal(\Qbar/\QQ)\to (\ZZ/n\ZZ)^\times$ is the mod $n$ cyclotomic character. 
Therefore, we have that 
\[
\det(\rho_{E_1,n}(\sigma)) = \det(\rho_{E_2,n}(\sigma)),
\]
and hence the result follows. 
\end{proof}

\subsection{Siegel functions}\label{subsec:Siegel}
The modular curves $X_G$ of genus 0 with $X_G(\mQ) \neq \emptyset$ are isomorphic to the projective line, and for each such curve, the function field is of the form $\mQ(h)$ for some modular function $h$ of level $N$. Giving the morphism $\pi_G$ is then equivalent to expressing the modular $j$-invariant in the form $J(h)$. Below we will describe how to compute this modular function $h$ using Siegel functions, but before doing so, we provide a {brief} introduction to Siegel functions. 
A reader interested in a full treatment of the topic should see \cite[Chapter 2]{kubertLang:modularunits}. 

\begin{defn}
The \cdef{Siegel function} $g_\textbf{a}(\tau)$ associated to $\textbf{a} = (a_1,a_2)\in(\QQ/\ZZ)^2$ is a function on the complex upper half plane $\mathbb{H}$ defined by 
$$g_\textbf{a}(\tau) \coloneqq -q_{\tau}^{(1/2)\textbf{B}_2(a_1)}e^{2\pi i a_2(a_1-1)/2}(1-q_z)\prod_{n=1}^\infty(1-q_\tau^nq_z)(1-q_\tau^n/q_z)$$
where $q_\tau=e^{2\pi i \tau}$, $z=a_1\tau+a_2$, $q_z=e^{2\pi i z}$ and $\textbf{B}_2(x)=x^2-x+\frac{1}{6}$ is the second Bernoulli polynomial.
\end{defn}

The utility of these functions is two fold. First, if $\textbf{a} \in (\ZZ\left[\frac{1}{N}\right]/\ZZ)^2$, then the divisor of $g_\textbf{a}$ is completely supported at the cusps of the modular curve $X(N)$ and is easily computable. Second, if once again we restrict to subscripts $\textbf{a} \in (\ZZ\left[\frac{1}{N}\right]/\ZZ)^2$, then there are explicit conditions under which products of these functions become modular functions for a congruence subgroup $\Gamma$ of level $N$. 
For example, see \cite[Chapter 2]{kubertLang:modularunits}, \cite[Section 2]{daniels:Siegelfunctions}, and \cite[Section 4]{sutherlandzywina2017primepower}. 

Below, we will give the two main theorems that will allow us to compute the $j$-maps for most of the genus zero modular curves with $-I$, but before we can state the theorems we need to establish some notation; here we will follow the notation from \cite[Section 4]{sutherlandzywina2017primepower}.

We let $\Gamma\subseteq\SL_2(\ZZ)$ be a congruence subgroup of level $N$ and let $P_1,\dots,P_r$ be the cusps of the modular curve $X_\Gamma$. For each cusp $P_i$, choose a representative $s_i\in\QQ\cup\{\infty\}$ and a matrix $A_i\in\SL_2(\ZZ)$ such that $A_i\cdot\infty = s_i$. For each $i$, let $w_i$ be the width of $P_i$ (i.e., the smallest integer such that $A_i\begin{psmallmatrix} 1&w_i\\0&1\end{psmallmatrix} A_i^{-1}$ is in $\Gamma$).

Again following \cite{sutherlandzywina2017primepower}, we let $\mathcal{A}_N$ be the subset of $\textbf{a} = (a_1,a_2)$ in $\ZZ\left[\frac{1}{N}\right]^2 \setminus \ZZ^2$ such that one of the following holds:
\begin{itemize}
  \item $0<a_1<1/2$ and $0\leq a_2<1$,
  \item $a_1=0$ and $0<a_2\leq 1/2$,
  \item $a_1=1/2$ and $0\leq a_2\leq 1/2$.
\end{itemize}
The set $\mathcal{A}_N$ is chosen so that every non-zero coset of $\ZZ [ \frac{1}{N} ]^2/\ZZ^2$ is represented by an element of the form $\textbf{a}$ or $-\textbf{a}$ for a unique $\textbf{a} \in\mathcal{A}_N$. 
The group $\SL_2(\ZZ)$ has a natural action on $\mathcal{A}_N$ given by $(\textbf{a},\gamma) = \textbf{a}\cdot\gamma$ where we consider $\textbf{a}$ as a row vector. 
Restricting this action to $\Gamma$, we can consider the $\Gamma$-orbits of $\mathcal{A}_N$. Given an orbit $\mathcal{O}$, we let 
\[
g_\mathcal{O}(\tau) = \prod_{\textbf{a}\in\mathcal{O}} g_\textbf{a}.
\]
From the work in \cite{kubertLang:modularunits}, we know that $g_\mathcal{O}^{12N}$ is a modular function for $\Gamma$. The hope is they can be used to generate functions on $X_\Gamma$.
  
\begin{lemma}\protect{\cite[Lemma 4.3]{sutherlandzywina2017primepower}\label{lem:divSFuncs}}
With the notation as above, we have
$${\rm div}(g_\mathcal{O}^{12N})= \sum_{i=1}^r\left(6Nw_j\sum_{\textbf{a}\in \mathcal{O}} B_2\left(\langle \left(\emph{\textbf{a}}A_j \right)_1 \rangle \right)\right)\cdot P_i,$$
where $\emph{\textbf{B}}_2(x)=x^2-x+\frac{1}{6}$,  $(\emph{\textbf{a}}A_j)_1$ is the first coordinate of the row vector $ \emph{\textbf{a}}A_j$, and $\langle x \rangle$ denotes the positive fractional part of $x$ (i.e., the number $0\leq\langle x \rangle<1$ such that $x- \langle x \rangle \in \ZZ$). 
\end{lemma}

Before we can state the next lemma, we need one last piece of notation. Let $\mathcal{O}_1,\dots,\mathcal{O}_n$ be the distinct $\Gamma$-orbits of $\mathcal{A}_N$ and let $D_i = {\rm div}(g_{\mathcal{O}_i}^{12N})$.

\begin{lemma}\protect{\cite[Lemma 4.4]{sutherlandzywina2017primepower}}\label{lem:haupt}
Suppose that $X_\Gamma$ is a genus 0 curve and that there is an $n$-tuple $\emph{\textbf{m}}=(m_1,\dots, m_n)\in\ZZ^n$ such that
$$\sum_{i=1}^n m_iD_i = -12N P_1 + 12NP_2.$$
Then there exists an explicitly computable $2N^2$-th root of unity $\zeta$ such that
$$h = \zeta\prod_{i=1}^n g_{\mathcal{O}}^{m_i}$$
is a hauptmodul for $\Gamma$.
\end{lemma}

We can use Lemma \ref{lem:haupt} to find generators for the function fields of each genus 0 modular curve $X_\Gamma$ where $-I\in \Gamma$. Once a hauptmodul $h$ is computed, it suffices to find an algebraic relationship between $h$ and the usual modular $j$-invariant $J$. 
This is done using the methods from \cite{daniels:Siegelfunctions}, which we now summarize.

We know that the function field of our modular curve is $\QQ(\zeta_N)(h)$ and that $J\in \QQ(\zeta_N)(h)$, and thus there is a rational function 
\[
f(t) = \frac{a_0+ a_1t +\cdots + a_kt^k}{b_0+ b_1t +\cdots +b_l t^l} \in \QQ(\zeta_N)(t)
\]
such that $f(h) = J$. We can find $f(t)$ by clearing the denominator to get 
\[a_0+ a_1h +\cdots + a_kh^k = b_0J+ b_1hJ +\cdots + b_l h^l J,\]
and since we know the $q$-expansions of $h$ and $J$ to as many places as necessary, we can turn this into a linear algebra problem. 
We do this by considering the functions in $S_1 = \{1, h, h^2, \dots,h^k\}$ and $S_2 = \{J, hJ, h^2J, \dots,h^{l}J\}$ as vectors by taking the coordinate 
vectors of some approximation with respect to the standard basis and then look for an intersection between ${\rm Span\,}( S_1)$ and ${\rm Span\,}( S_2)
$. 
After expanding $h$ and $J$ to sufficiently many places, we can find a common vector in these spans and we use it to find $f(t)$, and hence determine our $j$-map $\pi_{\Gamma}$.

For each of the groups $G\subseteq\GL_2(\ZZ/N\ZZ)$ with $-I\in G$, we can associate to it a congruence subgroup $\Gamma_G$ by letting $\Gamma_G = \pi^{-1}(H)$ where $H = G \cap \SL_2(\ZZ/N\Z)$ and $\pi\colon \SL_2(\ZZ) \to \SL_2(\ZZ/N\ZZ)$ is the standard component-wise reduction map. 
Note that in the construction of our congruence subgroup, we lose some information (i.e., we can have subgroups $G_1$ and $G_2$ of $\GL_2(\ZZ/N\ZZ)$ such that $G_1\cap\SL_2(\ZZ/N\ZZ) = G_2\cap\SL_2(\ZZ/N\ZZ)$). 
From our above discussion, we have that the modular curves $X_{\Gamma_G}$ and $X_G$ are isomorphic over $\QQ(\zeta_N)$. For an explicit example illustrating this point, we refer the reader to Example \ref{exam:Nn3}.

In order to find a model for $X_{\Gamma_G}$ that is isomorphic to $X_G$ over $\QQ$ and to ensure that we get the correct $j$-map for the given extension of $H = G \cap \SL_2(\ZZ/N\ZZ)$, we have to take an extra step to calibrate $h$ using rational moduli. 
We do this by precomposing $h$ with a fractional linear transformation that takes 3 known rational points on $X_G$ to the points 0, 1, and $\infty$. This yields a new hauptmodul $h_1$, and when we search for an algebraic relationship between $h_1$ and the classical $j$-function, we get a rational function defined over $\QQ$. It is admittedly messy, but using techniques from \cite{rouse2014elliptic}, we are able to find a transformation that greatly simplifies these models.

\section{\bf Group theoretic definitions of entanglements}
\label{sec:grouptheoreticdefinitions}
In this section, we define entanglements from a group theoretic perspective. 
The relationship between these notions and the entanglements of division fields mentioned in Section \ref{sec:Intro} will be postponed until Section \ref{sec:entanglementsellipticcurves}. 

\subsection*{Notation}
Let $G$ be a subgroup of $\GL_2(\Z/n\Z)$ for some $n\geq 2$ with surjective determinant, let $a < b$ be divisors of $n$, let $c= \lcm(a,b)$, and let $d=\gcd(a,b)$. 
Let $\pi_c\colon \GL_2(\zZ{n})\to \GL_2(\zZ{c})$ be the natural reduction map, and set $G_c := \pi_c(G)$. 
We have the following reduction maps and normal subgroups of $G_c$
\begin{align*}
& \pi_a\colon \GL_2(\zZ{c}) \to \GL_2(\zZ{a}), \quad & N_a := \ker(\pi_a) \cap G_c\\
& \pi_b\colon \GL_2(\zZ{c}) \to \GL_2(\zZ{b}), \quad & N_b := \ker(\pi_b) \cap G_c\\
& \pi_d\colon \GL_2(\zZ{c}) \to \GL_2(\zZ{d}), \quad & N_d := \ker(\pi_d) \cap G_c.
\end{align*}
We will abuse notation and denote restrictions of the above maps to subgroups of $\GL_2(\zZ{c})$ with $\pi_{a},\pi_b,$ and $\pi_d$. 

We now offer two equivalent definitions for when $G$ represents an $(a,b)$-entanglement.

\begin{defn}\label{def:representsentanglement}
We say that $G$ \cdef{represents an $(a,b)$-entanglement} if 
\[
\langle N_a,N_b\rangle \subsetneq N_d.\]
The \cdef{type} of the entanglement is the isomorphism type of the group $N_d/\langle N_a, N_b\rangle$. 
\end{defn}

\begin{lemma}\label{lemma:equivdefinitions}
The group $G$ represents an $(a,b)$-entanglement if and only if 
\[
N_d/\pi_a^{-1}(\pi_a(N_b)) \simeq N_d/\pi_b^{-1}(\pi_b(N_a)) \not\simeq \brk{I}.
\]
\end{lemma}

\begin{proof}
The equivalence follows from the equality $\pi_a^{-1}(\pi_a(N_b)) = \pi_b^{-1}(\pi_b(N_a)) = \langle N_a,N_b\rangle$. 
\end{proof}


\begin{remark}
As it will be useful later on, we define 
\[
N_{a,b}(G) := \pi_a^{-1}(\pi_a(N_b)) = \langle N_a,N_b\rangle = \pi_b^{-1}(\pi_b(N_a)). 
\]
In our computations, we keep track of this group as it will allow us to distinguish entanglements (see~ Example \ref{ex:NabG} for more details).

Definition \ref{def:representsentanglement} and the equivalent condition from Lemma \ref{lemma:equivdefinitions} represent different perspectives on the concept of groups representing an entanglement. The main definition is motivated by the number theory and clearly reflects the entanglement of the division fields of elliptic curves. The notion from Lemma \ref{lemma:equivdefinitions} takes a more group theoretic perspective, by using Goursat's lemma (\cite[page 75]{Lang:algebra} or \cite{Goursat}) to detect if there is an entanglement and the entanglement's type. 

To see how this comes from Goursat's lemma, one simply needs to consider the group $G_c$ as a subgroup of $G_a\times G_b$ by using the injection $g\mapsto (\pi_a(g),\pi_b(g))$. From this perspective, $G_c$ satisfies all the conditions necessary to apply Goursat's lemma which says that there are (constructible) $M_a \triangleleft G_a$ and $M_b\triangleleft G_b$  such that the image of $G_c$ in $G_a/M_a\times G_b/M_b$ is the graph of an isomorphism $G_a/M_a\simeq G_b/M_b$. In this context, the groups $M_a$ and $M_b$ are exactly $\pi_a^{-1}(\pi_a(N_b))$ and $\pi_b^{-1}(\pi_b(N_a))$, respectively. 

For further instances of the relationship between Goursat's lemma and entanglements, we refer the reader to \cite{brau3in2} and \cite[Lemma 8.2]{morrow2017composite}. 
\end{remark}

Definition \ref{def:representsentanglement} and Lemma \ref{lemma:equivdefinitions} provide us with a group theoretic way to describe entanglements.  
As mentioned in Section \ref{sec:Intro}, we are interested in studying when certain entanglements appear infinitely often. 
In order to make this question tractable, we need 
a definition which captures the notion of "maximal'' entanglements. 
The notion is subtle to define because we need to simultaneously capture when the entanglement is happening at the lowest possible level and when it has the largest possible type. 

Our definition reads as follows. 

%

\begin{defn}\label{defn:primitiveentanglement}
Consider the set 
\[
\mathcal{T}_G = \{ ( (a,b), H) \, | \, G \text{ represents an $(a,b)$-entanglement of type $H$} \}.
\] 
We define a relation on $\mathcal{T}_G$ by declaring that $ ((a_1,b_1), H_1) \leq ( (a_2,b_2), H_2 )$ if:
\begin{enumerate}
\item $H_1$ and $H_2$ are isomorphic and either $(a_2 \mid a_1$ and $b_2\mid b_1)$ or  $(b_2 \mid a_1$ and $a_2 \mid b_1)$, or
\item $H_1$ is isomorphic to a quotient of $H_2$ and either $(a_1 \mid a_2$ and $b_1 \mid b_2)$ or $(b_1 \mid a_2$ and $a_1 \mid b_2)$.
\end{enumerate}
We say the group $G$ represents a \cdef{primitive $(a,b)$-entanglement of type $H$} if $( (a,b), H )$ is the unique maximal element of $\mathcal{T}_G$ and $n = \lcm(a,b)$. 
\end{defn}

\begin{remark}[How to think of primitive entanglements]
Primitive $(a,b)$-entanglements should be thought of as fundamental building blocks of entanglements, and when $G$ represents a primitive $(a,b)$-entanglement of type $H$, one should think that $H$ is the largest type of entanglement that $G$ represents and $(a,b)$ is the lowest level that it occurs. 
This definition does not consider subgroups of $\GL_2(\zZ{n})$ to be primitive entanglements if they are the pre-images of entanglements of lower level or if they are constructed via fibered products of groups of smaller level. 
\end{remark}

\begin{remark}
\begin{enumerate}
\item []
\item The relation on $\mathcal{T}_G$ is not a partial order as it fails to be transitive; however, the relation is reflexive and anti-symmetric.  
\item If $\mathcal{T}_G$ does not have a maximal element with respect to $\leq$, then $G$ does not represent a primitive $(a,b)$-entanglement. 
\end{enumerate}
\end{remark}

To get a sense of the relation $\leq$, consider the following example. 

\begin{example}
Let $n = 420 = 2^2\cdot 3 \cdot 5 \cdot 7$ and let $G$ be a subgroup of $\GL_2(\zZ{n})$. 
\begin{itemize}
\item If $((2,3),\zZ{2})$ and $((4,3),\zZ{2})$ are in $\mathcal{T}_G$, then condition 1 in Definition \ref{defn:primitiveentanglement} says that $((4,3),\zZ{2}) \leq ((2,3),\zZ{2})$. 
\item If $((2,3),\zZ{2})$ and $((4,3),\zZ{4})$ are in $\mathcal{T}_G$, then condition 2 in Definition \ref{defn:primitiveentanglement} says that $((2,3),\zZ{2}) \leq ((4,3),\zZ{4})$. 
\item If $((2,3),\zZ{2} )$ and  $((5,7), \zZ{4})$ are in $\mathcal{T}_G$, then Definition \ref{defn:primitiveentanglement} says that these two pairs are incomparable.
\end{itemize}
\end{example}

For $p, q$ distinct primes, one can immediately see that any subgroup $G$ of $\GL_2(\zZ{pq})$ representing a $(p,q)$-entanglement must in fact be primitive. 

\begin{lemma}
If $G \subseteq \GL_2(\zZ{pq})$ represents a $(p,q)$-entanglement of type $H$, then $G$ represents a primitive $(p,q)$-entanglement of type $H$.
\end{lemma}

\begin{proof}
This follows immediately as the set $\mathcal{T}_G$ is a single element, namely $((p,q),H)$. 
\end{proof}

With the notion of primitive firmly established, we turn our attention to defining two classes of entanglements, which were mentioned in Section \ref{sec:Intro}. 

\begin{defn}\label{def:explained}
The group $G$ represents an \cdef{explained $(a,b)$-entanglement of type $T$} if $G$ represents a primitive $(a,b)$-entanglement of type $T$ and 
\[
[(\zZ{c})^\times : \det(N_{a,b}(G))] = [G_c:N_{a,b}(G)].
\]
\end{defn}

\begin{remark}\label{rmk:altexplained}
Alternatively, one could define an explained entanglement as follows. 
Let $N$ be the kernel of $\det\colon G_c\to (\ZZ/c\ZZ)^\times$. Then we have that $G$ represents an {explained $(a,b)$-entanglement} if $N \subseteq \langle N_a, N_b \rangle$. 
\end{remark}

\begin{defn}\label{def:unexplained}
The group $G$ represents an \cdef{unexplained $(a,b)$-entanglement of type $T$} if $G$ represents a primitive $(a,b)$-entanglement of type $T$ and 
\[
[(\zZ{c})^\times : \det(N_{a,b}(G))] \neq [G_c:N_{a,b}(G)].
\]
\end{defn}

\begin{remark}\label{rem:unexplainedtype}
When a group $G$ represents an unexplained entanglement, it will be useful to describe its type as a pair $(T,T')$, where $G$ represents a primitive $(a,b)$-entanglement of type $T$ and $G' = G \cap \SL_2(\zZ{n})$ represents a primitive $(a,b)$-entanglement of type $T'$. 
With this notation, the group $T$ corresponds to the type of the total entanglement, and the group $T'$ corresponds to the type of the entanglement which is ``totally" unexplained. See Example \ref{exam:BJ} for further discussion. 
\end{remark}

\begin{remark}
All of the definitions in this section can be extended to open subgroups $G$ of $\GL_2(\widehat{\mZ})$ with surjective determinant. 
To do this, we first consider the level of the subgroup $G$, which is the the least positive integer $N$ such that $G$ is the inverse image of its image under the reduction map $\GL_2(\widehat{\mZ}) \to \GL_2(\zZ{N})$. 
If we let $G_N \subseteq \GL_2(\zZ{N})$ denote the image of $G$ under this reduction map and let $a< b$ be divisors of $N$, then we can say that $G\subseteq \GL_2(\widehat{\mZ})$ represents an $(a,b)$-entanglement when $G_N$ represents an $(a,b)$-entanglement, and we can similarly define the notions of primitive, explained, and unexplained entanglements. 
We avoid defining these notions for open subgroups $\GL_2(\widehat{\mZ})$ with surjective determinant in order to emphasize that our analysis and computations occur at finite level. 
\end{remark}

\section{\bf Entanglements for images of Galois associated to elliptic curves}
\label{sec:entanglementsellipticcurves}
In this section, we explain how the definitions from Section \ref{sec:grouptheoreticdefinitions} codify the entanglements mentioned in Section \ref{sec:Intro}. 

\subsection*{Notation}
Let $E/\mQ$ be an elliptic curve and for a positive integer $n\geq 2$, let $G = \Im(\rho_{E,n}) \subseteq \GL_2(\zZ{n})$ denote the mod $n$ image of Galois. 
Recall that the $n$-division field $\mQ(E[n])$ is the fixed field of $\overline{\mQ}$ by the kernel of the mod $n$ representation, and so the Galois group of this number field is the image of the mod $n$ representation.

We now recall the notation established in Section \ref{sec:grouptheoreticdefinitions}. 
Let $a < b$ be proper divisors of $n$, and let $d=\gcd(a,b)$ and $c = \lcm(a,b)$. 
Let $\pi_c\colon \GL_2(\zZ{n})\to \GL_2(\zZ{c})$, and set $G_c := \pi_c(G)$. 
We have the following reduction maps and normal subgroups of $G_c$
\begin{align*}
& \pi_a\colon \GL_2(\zZ{c}) \to \GL_2(\zZ{a}), \quad & N_a := \ker(\pi_a) \cap G_c\\
& \pi_b\colon \GL_2(\zZ{c}) \to \GL_2(\zZ{b}), \quad & N_b := \ker(\pi_b) \cap G_c\\
& \pi_d\colon \GL_2(\zZ{c}) \to \GL_2(\zZ{d}), \quad & N_d := \ker(\pi_d) \cap G_c.
\end{align*}
Again, we will abuse notation and denote restrictions of the above maps to subgroups of $\GL_2(\zZ{c})$ with $\pi_{a},\pi_b,$ and $\pi_d$.

We summarize the Galois correspondence between the $a,b,c,d$-division fields and the groups $N_a,N_b,G_c,N_d$ in Figure \ref{figure:Galoiscorrespondence}.
\begin{figure}[h!]
\[
\xymatrix@C=.1em{
 &\{I\}\ar@{-}[dl]\ar@{-}[dr] & & & && & & & &\mQ(E[c]) \ar@{-}[dl]\ar@{-}[dr]& \\
 N_a & &N_b&  &&& & & & \mQ(E[a]) & & \mQ(E[b])\\
 &\langle N_a, N_b \rangle\ar@{-}[ul]\ar@{-}[ur] & & &&& & & & &\mQ(E[a])\cap \mQ(E[b]) \ar@{-}[ul]\ar@{-}[ur]& \\
 &N_d \ar@{-}[u] & & & &&& & & &\mQ(E[d])\ar@{-}[u] & \\
 & G_c\ar@{-}[u]& & & &&& & & & \QQ\ar@{-}[u]& \\
}
\]
\caption{Galois correspondence for various division fields}
\label{figure:Galoiscorrespondence}
\end{figure}
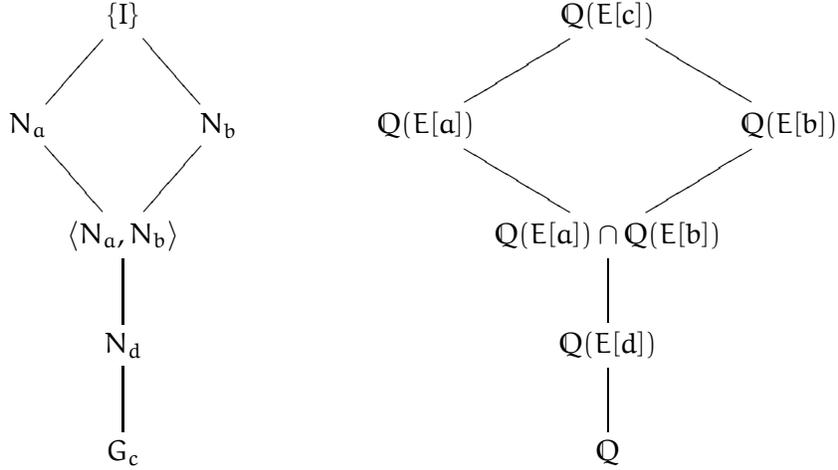

Recall our previous notion of entanglement from Definition \ref{defn:firstentanglement}. As a consequence of the Galois correspondence and our definitions from Section \ref{sec:grouptheoreticdefinitions}, we have the following lemma.

\begin{lemma}\label{lem:connect_groups_curves}
Let $E/\QQ$ be an elliptic curve, and let $a<b$ be positive integers. The group $\Im(\rho_{E,ab})$ represents an $(a,b)$-entanglement of type $T$ if and only if $E$ has an $(a,b)$-entanglement of type $T$. 
\end{lemma}

Using Lemma \ref{lem:connect_groups_curves} and the notions from Section \ref{sec:grouptheoreticdefinitions}, we can define various type of entanglements for elliptic curves.

\begin{defn}\label{def:ellipticentanglement}
We say that an elliptic curve $E/\mQ$ has an \cdef{$(a,b)$-entanglement of type $T$} if for some $n\geq 2$ and proper divisors $a < b$, the mod $n$ image of Galois $\,\Im(\rho_{E,n})$ represents an $(a,b)$-entanglement of type $T$.
\end{defn}

\begin{defn}\label{def:ellipticexplained}
We say that an elliptic curve $E/\mQ$ has an \cdef{explained $(a,b)$-entanglement of type $T$} if for some $n\geq 2$ and proper divisors $a < b$, the mod $n$ image of Galois $\,\Im(\rho_{E,n})$ represents an explained $(a,b)$-entanglement of type $T$.
\end{defn}

\begin{remark}
We now explain how Definition \ref{def:ellipticexplained} encapsulates the explained entanglements we discussed in Section \ref{sec:Intro}. 
Let $E/\Q$ be an elliptic curve and let $G$ be the mod $n$ image of Galois. 
Notice that the group $N $ in Remark \ref{rmk:altexplained} correspond to groups with fixed field $\mQ(\zeta_c)$. 
Indeed, since $\det\circ\rho_{E,c}\colon \Gal(\Qbar/\QQ)\to(\ZZ/c\ZZ)^\times$, we have that $\QQ(\zeta_c)$ is fixed by the elements with determinant 1.
Therefore, if $G$ represents an explained $(a,b)$-entanglement, we have that the intersection of $\Q(E[a])\cap\Q(E[b])$ is the compositum of $\Q(E[a])\cap\Q(\zeta_b)$ and $\Q(E[b])\cap\Q(\zeta_a)$. 
In this case, the entanglement is explained by the Kronecker--Weber theorem and the Weil pairing.  
\end{remark}

\begin{defn}\label{def:elliptic_unexplained}
We say that an elliptic curve $E/\mQ$ has an \cdef{unexplained $(a,b)$-entanglement of type $T$} if for some $n\geq 2$ and proper divisors $a < b$, the mod $n$ image of Galois $\,\Im(\rho_{E,n})$ represents an unexplained $(a,b)$-entanglement of type $T$.
\end{defn}

To conclude this section, we discuss several examples of unexplained entanglements.  

\begin{example}
The elliptic curve from Example \ref{exam:firstunexplainedentanglement} has an unexplained $(2,3)$-entanglement of type $\zZ{2}$. 
One could say that this elliptic curve has a ``totally" unexplained $(2,3)$-entanglement of type $(\zZ{2},\zZ{2})$ (cf.~Remark \ref{rem:unexplainedtype}) as no part of the entanglement is lost when intersecting the mod $6$ image of Galois with $\SL_2(\zZ{6})$. 
More precisely, the groups $\Im(\rho_{E,6})$ and $\Im(\rho_{E,6}) \cap \SL_2(\zZ{6})$ both represent the same $(2,3)$-entanglement of type $\zZ{2}$. 
\end{example}

\begin{example}\label{exam:BJ}
By the Hilbert irreducibility theorem, most of the members of the family of elliptic curves from Brau--Jones \cite{brau3in2} have an unexplained $(2,3)$-entanglement of type $(S_3,\zZ{3})$. 
Indeed, the authors classify elliptic curves $E/\mQ$ satisfying $\mQ(E[2]) \subseteq \mQ(\zeta_3,\Delta_E^{1/3}) \subseteq \mQ(E[3])$. 
When $3$ is not an exceptional prime for $E$, there is only one quadratic subfield of $\mQ(E[3]$), namely $\mQ(\sqrt{-3})$, and when $2$ is not exceptional for $E$ and $\mQ(E[2]) \subseteq \mQ(E[3])$, it must be that $\Delta_E = -3t^2$ for some $t\in \mQ^\times$. In this case, part of the $(2,3)$-entanglement is ``explained"  by $\mQ(\sqrt{-3})$, and so $E$ has an unexplained $(2,3)$-entanglement of type $(S_3,\zZ{3})$. 
\end{example}

\begin{example}\label{Ex:MaxGroups}
The purpose of this example is to illustrate why we are considering only the maximal groups (partially ordered by containment, up to conjugation) representing a primitive entanglement of a given level and type. We will show how these maximal groups capture all of the information that is {essential} to understanding the moduli space of elliptic curves with entanglements of this kind and give anecdotal justification for {\sc Step 2} in our proof of Theorem \ref{thm:main}. 

To this end, we search for all the the admissible groups $G\subseteq\GL_2(\ZZ/6\ZZ)$ that represent a $(2,3)$-entanglement of type $(S_3,\zZ{3})$. Our search yield four such groups
\begin{align*}
G_1 &=\left\langle 
\begin{pmatrix}
    5& 1\\
    4& 3
\end{pmatrix},
\begin{pmatrix}
    4& 1\\
    1& 0
\end{pmatrix}
\right\rangle, \phantom{\hbox{a}}
G_2 =\left\langle  
\begin{pmatrix}
    4& 1\\
    5& 3
\end{pmatrix},
\begin{pmatrix}
    2& 3\\
    1& 4
\end{pmatrix}
\right\rangle,  \\
G_3 &=\left\langle  
\begin{pmatrix}
    2& 5\\
    1& 3
\end{pmatrix},
\begin{pmatrix}
    4& 3\\
    5& 2
\end{pmatrix}
\right\rangle,  \phantom{\hbox{a}}
G_4 =\left\langle
\begin{pmatrix}  
    2& 5\\
    1& 3
\end{pmatrix},
\begin{pmatrix}
    1& 1\\
    0& 5
\end{pmatrix}
\right\rangle,
\end{align*}
which have size 48, 12, 6, and 6 respectively. Further, for each $i$, the modular curve $X_{G_i}$ is a $\mP^1$.

Up to conjugation, $G_1$ contains the other three groups; in fact, $G_2 = \langle-I,G_3 \rangle = \langle-I,G_4 \rangle$ and $G_2 = G_1 \cap \pi_3^{-1}(3{\texttt{B}})$, where $\pi_3$ is the map $\pi_3\colon \GL_2(\ZZ/6\ZZ)\to \GL_2(\ZZ/3\ZZ)$. 
In particular, we see that the rational points on $X_{G_2}$ correspond to elliptic curves with mod 6 image representing a $(2,3)$-entanglement of type $(S_3,\zZ{3})$ (since $G_2 \subseteq G_1$) \emph{and} with a 3-isogeny (since $G_2 \subseteq \pi_3^{-1}(3{\texttt{B}})$). The groups $G_3$ and $G_4$ then occur as particular twists inside each of the $\Qbar$-isomorphism classes of elliptic curves with image in $G_2$. 

On the other hand, the modular curve $X_{G_1}$ parametrizes $\Qbar$-isomorphism classes of elliptic curves with mod 6 image of Galois \emph{contained} in $G_1$, and so if $E/\QQ$ is an elliptic curve with a $(2,3)$-entanglement of type $(S_3,\zZ{3})$, then $E$ corresponds to a rational point on $X_{G_1}$ regardless of whether or not the elliptic curve has any additional algebraic structure. 
Therefore, the modular curve $X_{G_1}$ is the fundamental object as far as $(2,3)$-entanglements of $(S_3,\zZ{3})$-type is concerned. Once $X_{G_1}$ is computed, to understand the finer question of what other structures can occur along with an entanglement, one can use fibered products of modular curves (cf.~\cite[Section 8.4]{morrow2017composite}). 

\end{example}

\begin{example}\label{ex:NabG}
Before proceeding, we provide an example illustrating the way in which we can obtain information about an entanglement group from knowing the group $N_{a,b}(G)$.

Consider the two subgroups of $\GL_2(\ZZ/6\ZZ)$ given by
\[
\begin{array}{lcr}
G_1 = \left\langle \begin{pmatrix} 2&5 \\3 &2 \end{pmatrix},\begin{pmatrix}1 & 3\\3 &2 \end{pmatrix}\right\rangle  & \text {and } & 
G_2 = \left\langle \begin{pmatrix} 5&5 \\ 0& 5\end{pmatrix},\begin{pmatrix} 2& 5\\3 &2 \end{pmatrix},\begin{pmatrix} 2&1 \\3 &1 \end{pmatrix}  \right\rangle.
\end{array}
\]
These are two of the five maximal groups in the set of all level 6 groups representing an unexplained $(2,3)$-entanglement of type $\ZZ/2\ZZ$. That is, they are not conjugate and there are no level 6 groups that represent an $(2,3)$-entanglement of type $\ZZ/2\ZZ$ that contain either of these groups, up to conjugation. Both of these groups also have the property that their mod 2 image is full and their image mod 3 is conjugate to $\texttt{3B}$\footnote{All of the remaining three maximal groups  have full image mod 2. One of the groups has image conjugate to $\texttt{3Ns}$ mod 3, and the remaining two group have mod 3 image conjugate to $\texttt{3Nn}$ and can be found in Example \ref{exam:Nn3}.}. For each group we compute
\[
\begin{array}{lcr}
N_{2,3}(G_1) = \left\langle \begin{pmatrix}1 &2 \\0 & 1\end{pmatrix},\begin{pmatrix} 1&0 \\ 0& 5\end{pmatrix},\begin{pmatrix} 4&3 \\ 3&1 \end{pmatrix}   \right\rangle & \hbox{ and } & N_{2,3}(G_2) = \left\langle \begin{pmatrix} 1&2 \\0 &1 \end{pmatrix},\begin{pmatrix} 5&0 \\0 &1 \end{pmatrix},\begin{pmatrix} 1& 3\\ 3& 4\end{pmatrix} \right\rangle.
\end{array}
\]

Inspecting these groups, we see that
\[
\begin{array}{lcr}
\pi_3(N_{2,3}(G_1)) = \left\langle \begin{pmatrix} 1& 0\\0 &2 \end{pmatrix},\begin{pmatrix} 1& 1\\ 0& 1\end{pmatrix} \right\rangle & \hbox{ and } & \pi_3(N_{2,3}(G_2)) = \left\langle \begin{pmatrix} 2&0 \\ 0& 1\end{pmatrix},\begin{pmatrix} 1&1 \\ 0&1 \end{pmatrix} \right\rangle.
\end{array}
\]
Thus, if $E$ is an elliptic curve with $\Im(\rho_{E,6}) = G_1$, we know that the intersection between $\QQ(E[2])$ and $\QQ(E[3])$ is the fixed field of $\pi_3(N_{2,3}(G_1))$, which is exactly the field of definition of the 3-isogeny that $E$ must have (since $\pi_3(N_{2,3}(G_1))$ only has 1's in the upper left hand corner). On the other hand, if $\Im(\rho_{E,6}) = G_2$, then the intersection occurs between the other quadratic field in $\QQ(E[3])$ that is not $\QQ(\sqrt{-3})$. It is worth pointing out here that we know that neither of these groups fix $\QQ(\zeta_3)$ since $\det(\pi_3(N_{2,3}(G_i))) = (\ZZ/3\ZZ)^\times$ for $i = 1,2$.

\end{example}

\section{\bf Computations of modular curves}\label{sec:computataions}
\label{sec:computationsmodularcurves}
In this section, we begin our classification of unexplained $(p,q)$-entanglements which occur for infinitely many isomorphism classes of elliptic curves where $p\neq q$ are primes. 
To begin, we show that there are only finitely many groups representing such entanglements. 

\begin{lemma}\label{lemma:levelpq}
Suppose that $G\subseteq \GL_2(\ZZ/pq\ZZ)$ represents an {unexplained} $(p,q)$-entanglement, and let $H = G\cap\SL_2(\ZZ/pq\ZZ)$. 
Then, there exists a congruence subgroup $\Gamma\subseteq\SL_2(\mZ)$ of level $pq$ with $\pi_{pq}(\Gamma) = H$. 
\end{lemma}

\begin{proof}
Since the $\SL_2$-level of $G$ (i.e., the level of $H$) divides the $\GL_2$-level, it suffices to show that the $\SL_2$-level of $G$ is does not divide $p$ or $q$.
Since $G$ represents an unexplained entanglement, we have that $H $ is not all of $\SL_2(\ZZ/pq\ZZ)$, and hence the $\SL_2$-level of $G$ is not 1. 
Suppose for the sake of contradiction that the $\SL_2$-level of $G$ equal $p$. 
Then, $H = S(p) \times \SL_2(\ZZ/q\ZZ)$, and it follows that $\SL_2(\ZZ/q\ZZ) \subset N_q$ (the kernel of the restriction map coming from the $q$ side in the fibered product implicit in $G \subset \GL_2(\ZZ/p\ZZ) \times \GL_2(\ZZ/q\ZZ)$). This implies that the entanglement field must be a subfield of $\mQ(\zeta_q)$, which contradicts our initial assumption that $G$ represents an unexplained entanglement. 
A similar proof also shows that the $\SL_2$-level of $G$ cannot equal $q$, and therefore the $\SL_2$-level of $G$ is equal to $pq$, as desired.
\end{proof}

\begin{prop}\label{prop:finite_search}
There are only finitely many admissible subgroups $G\subseteq \GL_2(\ZZ/pq\ZZ)$ representing a unexplained $(p,q)$-entanglement that have genus $\leq 1$ as $p$ and $q$ vary over pairs of distinct primes. 
\end{prop}

\begin{proof}
Suppose that $G$ is an admissible subgroup of $\GL_2(\ZZ/pq\ZZ)$ which represents an unexplained $(p,q)$-entanglement with genus less than or equal to 1. 
Since $G$ represents an \emph{unexplained} entanglement, Lemma \ref{lemma:levelpq} tells us that  there exists a congruence subgroup $\Gamma\subseteq\SL_2(\mZ)$ of level $pq$ with $\pi_{pq}(\Gamma) = H$ such that the genus of the modular curve $X_\Gamma$ is less than or equal to 1. 
By Cox--Parry \cite{CoxParry}, there is an upper bound on the level of such groups $H$ (i.e., the product $pq$ is bounded) and by Cummins--Pauli \cite{CumminsPauli}, we have a complete list of possible $H$. 
Since the possible levels are bounded and at each level there are finitely many possible admissible subgroups, the result follows. 
\end{proof}

With Proposition \ref{prop:finite_search} and \cite{CoxParry, CumminsPauli}, we now know that if $G$ represents an unexplained $(p,q)$-entanglement and the genus of $X_G$ is at most $ 1$, then $G$ has level less than or equal to $39$.
Therefore, in order to prove Theorem \ref{thm:main}, we only need to search the levels $N = p\cdot q$ where $N\leq 39$. 

\begin{remark}
Of course, we could attempt to classify unexplained $(p^n,q^m)$-entanglements, however many of the computational aspects become very complicated and tedious (cf.~Remark \ref{rem:restrict}).
One other subtlety with classifying unexplained $(p^n,q^m)$-entanglements is that the group $G\cap \SL_2(\zZ{p^nq^m})$ may be the preimage under the reduction map of a subgroup $\SL_2(\zZ{a})$ where $a\mid p^nq^m$. We refer the reader to \cite[Section 3]{JonesM:Nonabelianentanglements} for further discussion.

At this time, we also note that by Definitions \ref{defn:firstentanglement} and \ref{def:ellipticentanglement} there cannot exist a \textit{non-trivial} $(p^n,p^m)$-entanglement where $1 \leq n\leq m$. 
When investigating $p^n$- and $p^m$-division fields, entanglements are not the correct relationship to study, rather one should ask when $\mQ(E[p^n]) = \mQ(E[p^m])$. 
This question has been considered by Rouse--Zureick-Brown \cite[Remark 1.6]{rouse2014elliptic} and Lozano-Robledo and the first author \cite[Theorem 1.4]{danielsLR:coincidences}. 
\end{remark}

For the remainder of the section, we provide examples of computing genus 0 modular curves associated to groups representing an unexplained entanglement. 
When the group contains $-I$, we use the Siegel functions method described in Subsection \ref{subsec:Siegel}, and for the curves whose groups do not have $-I$, we compute the models for these curves by hand using interesting relationships between them. 

\subsection{A genus zero example via Siegel functions}
The goal of this section is to carefully compute the $j$-map of one genus 0 modular curve using Siegel functions. 
Let $G$ be the genus 0, admissible subgroup of $\GL_2(\zZ{14})$ generated by
\[
G := \left\langle 
\begin{pmatrix}
10&7 \\ 5&11
\end{pmatrix},
\begin{pmatrix}
2&7 \\ 3&3
\end{pmatrix},
\begin{pmatrix}
5&7 \\ 9&12
\end{pmatrix}
\right\rangle 
\]
and let $H = G\cap\SL_2(\ZZ/14\ZZ).$

The group $G$ represents an unexplained $(2,7)$-entanglement of type $\ZZ/3\ZZ$. By computing the subgroup $N_{2,7}(G)$, we have that for a given elliptic curve $E/\Q$ with $\rho_{E,14}(G_{\mQ})$ conjugate to $G$, $\mQ(E[2]) \cap \mQ(E[7]) = \mQ(x(P))$ where $\langle P \rangle$ is the kernel of a $7$-isogeny $E\to E'$ and $x(P)$ corresponds to the $x$-coordinate of $P$. In other words, we see that the entanglement field is contained inside of the kernel of a $7$-isogeny.  

Recall the notation established in Subsection \ref{subsec:Siegel}. 
The first step is to compute the orbits of $H$ acting on $\mathcal{A}_{14}$. 
There are exactly 9 orbits that are represented by a unique element in the set $$S = \left\{
    \left( \frac{1}{14}, 0 \right),
    \left( \frac{1}{7}, 0 \right),
    \left( \frac{3}{14}, 0 \right),
    \left( \frac{5}{14}, 0 \right),
    \left( \frac{1}{2}, 0 \right),
    \left( 0, \frac{1}{14} \right),
    \left( \frac{1}{14}, \frac{1}{14} \right),
    \left( 0, \frac{1}{7} \right),
    \left( 0, \frac{3}{14} \right)
\right\}.$$
Computing the divisors for the product of Siegel functions as in Lemma \ref{lem:divSFuncs}, we see that if $\mathcal{O}$ is the orbit of $\mathcal{A}_{14}$ represented by $\left( \frac{3}{14}, 0 \right),$
then the function $g_{ \mathcal{O} }^{12\cdot 14}(\tau)$ has divisor $-12\cdot 14 P_1 + 12\cdot 14 P_2$, where $P_1$ is the point at infinity and $P_2$ is the cusp represented by the rational number 1. 
Therefore, by Lemma \ref{lem:haupt}, we have that 
$$h(\tau) = g_\mathcal{O}(\tau)$$
is a hauptmodul for the modular curve corresponding to the congruence subgroup associated to $H$. 
In this case, we have the following $q$-expansion
\begin{align*}
h &= q^{-1/14} -1 - q^{1/7} + q^{2/7} + q^{5/14} - q^{5/7} - q^{11/14} + q^{6/7} + 2q^{13/14} + O(q),
\end{align*}
where $q = e^{2\pi i \tau}$.

Next, we attempt to write the usual $j$-map's $q$-expansion as a rational function in the $q$-expansion of $h$ using linear algebra. 
A \texttt{Magma} computation shows that if $j(\tau)$ is the usual modular $j$-function and 
\[J_1(t) = \footnotesize \frac{(t^2 + 3t + 3)^3(t^6 + 7t^5 + 22t^4 + 47t^3 + 76t^2 + 75t + 31)
P_1(t)^3}{(t + 1)^{14}
(t + 2)^14
(t^3 + 4t^2 + 3t - 1)^2},\]
where 
\begin{small}
\[
P_1(t) = (t^{12} + 16t^{11} + 113t^{10} + 466t^9 + 1254t^8 + 2334t^7 + 3091t^6 + 2886t^5 + 1752t^4 + 532t^3 + 11t^2 + 24t + 37),
\]
\end{small}
then $J_1(h) = j(\tau)$ (i.e., that $J_1$ is the $j$-map for $X_H$). 
Moreover, we have that an elliptic curve $E$ over $\QQ(\zeta_{14})$ has mod 14 image of Galois contained in $H$ if and only if $j(E)$ is in $J_1(\QQ(\zeta_{14}))$. 

In order to recover the full group $G$, we need to calibrate our function $h$.
To do this, we compose $h$ with a fractional linear transformation that takes the pre-image of three known rational points to 0, 1, and $\infty$. For this example, we use the $j$-invariants 
\[
\frac{51181724570498001}{4}, \frac{5841700537729}{36}, \text{ and } \frac{16997034248155273645704721}{141745549885174404}
\]
which we found by hand. 

The result is a new hauptmodul, but its $q$-expansion and the corresponding new $j$-map are too messy to reproduce here. 
Using the methods in \cite{rouse2014elliptic}, we are able to find a nicer parameterization, which yields the $j$-map
$$J(t) = \frac{
(t^2 - t + 1)^3
(t^6 - 5t^5 + 12t^4 - 9t^3 + 2t^2 - t + 1)
P_2(t)^3}{
(t - 1)^2t^2(t^3 - 2t^2 - t + 1)
}$$
with 
\begin{small}
\[
P_2(t) =(t^{12} - 8t^{11} + 265t^{10} - 1474t^9 + 5046t^8 - 10050t^7 + 11263t^6 - 7206t^5 + 2880t^4 - 956t^3 + 243t^2 - 4t + 1).
\]
\end{small}
\vspace*{-1em}
\begin{remark}[Reality checks]
To check that the above $j$-map $J(t)$ is correct, we perform the following checks. 
Although, $J_1(t)$ and $J(t)$ parameterize different genus 0 curves over $\QQ$, we check that they give the same genus 0 curve when considered over $\QQ(\zeta_{14})$. 
Second, we define the generic elliptic curve $\mathcal{E}_t/\mQ(t)$ with $j$-invariant $J(t)$ and check that it has the correct entanglement. 
\end{remark}

\subsection{Genus zero examples without $-I$.}\label{subsec:constellation}
In this subsection, we will discuss some of the genus zero examples without $-I$.
To begin, we recall that in \cite[Section 5]{rouse2014elliptic}, the authors give a detailed discussion how to compute, for a genus zero subgroup $H$ without $-I$, a family of curves $E_t$ over an open subset $U\subseteq \mP^1$ such that an elliptic curve $E/\mQ$ without CM has image of Galois contained in a subgroup conjugate to $H$ if and only if there exists $ t \in U(\mQ)$ such that $E_t \simeq E$. 
In general, performing this computation boils down to finding the correct quadratic twist of the generic elliptic curve over $\mQ(t)$ with prescribed image of Galois, which can be quite computationally expensive.

In our situation, we are able to use relationships between entanglements (cf.~Examples \ref{exam:firstunexplainedentanglement} and \ref{exam:dual}) to form a constellation (see Figure \ref{fig:constellation}) between these groups, which greatly assists in determining their \textit{fine} moduli spaces. 
By a constellation, we mean that these unexplained entanglements can be connected through isogenies and twisting, which we will illustrate by an example. 

Consulting our search, we see that there are exactly 2 maximal groups of label \texttt{[GL2, 5B.4.1]} and two maximal groups of label \texttt{[GL2, 5B.4.2]}, and all four of these groups represent an unexplained $(2,5)$-entanglement of type $\zZ{2}$. 
Of the 2 groups with label \texttt{[GL2, 5B.4.1]}, we denote by $G_1$ the one that represents an entanglement which is contained inside of the kernel of the 5-isogeny (i.e., given an elliptic curve $E/\mQ$ with $\rho_{E,10}(G_{\mQ})$ conjugate to $G_1$, we have that $\mQ(E[2])\cap \mQ(E[5]) = \mQ(x(P))$ where $\langle P \rangle$ is the kernel of a 5-isogeny $E\to E'$ and $x(P)$ denotes the $x$-coordinate of $P$) and the other group with label \texttt{[GL2, 5B.4.1]} by $G_4$. 
We see that there is exactly one group, call it $G_2$, with label \texttt{[GL2, 5B.4.2]}, which is related to $G_1$ by Lemma \ref{lem:isogrep}. The remaining group with label \texttt{[GL2,5B.4.2]} is denoted by $G_3$. 

With this, we have that the 4 subgroups of $\GL_2(\ZZ/10\ZZ)$ that are of interest are
\begin{align*}
G_1 &=\left\langle\
\begin{pmatrix}
6 & 5\\ 7& 1
\end{pmatrix},
\begin{pmatrix}
7& 0\\ 9& 9
\end{pmatrix},
\begin{pmatrix}
4& 5\\ 5& 4
\end{pmatrix}
 \right\rangle,
\phantom{\ \ and\ \ }G_2 =\left\langle\
\begin{pmatrix}
1& 5\\ 1& 6
\end{pmatrix},
\begin{pmatrix}
1& 0\\ 0& 7
\end{pmatrix},
\begin{pmatrix}
4& 5\\ 5& 4
\end{pmatrix}
 \right\rangle,\\
G_3 &= \left\langle\
\begin{pmatrix}
1& 5\\ 1& 6
\end{pmatrix},
\begin{pmatrix}
1& 0\\ 5& 3
\end{pmatrix},
\begin{pmatrix}
4& 5\\ 5& 4
\end{pmatrix}
 \right\rangle, \phantom{\ \ and\ \ }
G_4 = \left\langle\
\begin{pmatrix}
6& 5\\ 7& 1
\end{pmatrix},
\begin{pmatrix}
8& 5\\ 7& 6
\end{pmatrix},
\begin{pmatrix}
9& 5\\ 0& 9
\end{pmatrix}
 \right\rangle.
\end{align*}

We can find elliptic curves over $\mQ$ whose mod 10 images of Galois corresponds to $G_i$ for $i = 1,\dots, 4$. 
Let $E_1$, $E_2$, $E_3$, and $E_4$ be the elliptic curves with Cremona labels \href{https://www.lmfdb.org/EllipticCurve/Q/193600hc1/}{\texttt{193600hc1}}, \href{https://www.lmfdb.org/EllipticCurve/Q/193600hc2/}{\texttt{193600hc2}}, \href{https://www.lmfdb.org/EllipticCurve/Q/38720l2}{\texttt{38720l2}}, and \href{https://www.lmfdb.org/EllipticCurve/Q/38720l1}{\texttt{38720l1}}, respectively. 
We start by observing from the Cremona labels that $E_1$ and $E_2$ (resp.~$E_3$ and $E_4$) are 5-isogenous. 
Looking closer, we see that $E_1$ and $E_4$ (resp.~$E_2$ and $E_3$) are related through twisting by 5. 
We also note that for each $i\in\{1,2,3,4\}$ the discriminant $\Delta_{E_i}$ of $E_i$ is congruent to $-110$ modulo rational squares, and so $\QQ(\sqrt{-110})\subseteq \QQ(E_i[2])$.

Next, we observe that $E_1$ has a point of order 5 defined over $\QQ(\sqrt{-110})$, namely the point $P_1 = (-1503,4400\sqrt{-110})$. 
From this and a \texttt{\texttt{Magma}} computation, we see that $E_1$ has a $(2,5)$-entanglement of type $\ZZ/2\ZZ$, no other entanglements of this level, and the common field between the 2- and 5-division field is contained inside the field of definition of the kernel of the 5-isogeny connecting $E_1$ and $E_2$. 
Here we have that $\rho_{E_1,5}$ has image contained in the group with Sutherland label \texttt{5B.4.1} while $E_2$ has related image \texttt{5B.4.2}.
As such, $E_2$ has a point $P_2$ of order 5 defined over a $\ZZ/4\ZZ$-extension of $\QQ$ that contains $\sqrt{5}$. 
Analyzing the 5-division field of $E_2$ further, we see that $\QQ(\sqrt{-110})$ is also in $\QQ(E_2[5])$ and since once again $\Delta_{E_2} \equiv -110 \bmod (\QQ^\times)^2$, $E_2$ also has a nontrivial $(2,5)$-entanglement of type $\zZ{2}$.

Twisting $E_2$ by $5$, we get the curve $E_3$. A simple computation shows that again $\QQ(\sqrt{-110}) \subseteq \QQ(E_3[5])$, but this time it is playing the role of the other quadratic subfield inside the division field. That is to say, if we fix bases for $E_2[5]$ and $E_3[5]$ that are compatible with twisting by 5, then the two quadratic extensions inside $\QQ(E_2[5])$ and $\QQ(E_3[5])$ that are not $\QQ(\sqrt{5})$, in this case $\QQ(\sqrt{-110})$ and $\QQ(\sqrt{-550})$, are each identified with an index two subgroup of the common group $\Im(\rho_{E_2,5}) \simeq \Im(\rho_{E_3,5})$. Since we chose compatible bases, we can compare which index two subgroup $\QQ(\sqrt{-110})$ corresponds to inside of the common group $\Im(\rho_{E_2,5})\simeq \Im(\rho_{E_3,5})$. The point is, these two curves have the same 5-division field, but the common subfields are identified with different subgroups of their (common) mod 5 image; this is the difference between the $(2,5)$-entanglements that $E_2$ and $E_3$ have. 
The group theoretic approach is sensitive enough to detect this and distinguish them.

The last curve to look at is $E_4$. In this case, $\Im(\rho_{E_4,5})$ is the group \texttt{5B.4.1}, but the point of order 5 defined over a quadratic field is $P_4 = (-1181 , 8800\sqrt{-22})$. Although we do not have that $\QQ(\sqrt{\Delta_{E_4}})$ is the same as the field of definition of this point, we do still have that $\QQ(\sqrt{-110})\subseteq \QQ(\sqrt{5},\sqrt{-22})\subseteq \QQ(E[5])$. Thus, this curve still exhibits a $(2,5)$-entanglement of type $\ZZ/2\ZZ$. Note that the mod 5 image of $E_4$ is the same as that of $E_1$, but the entanglement occurs between structurally different fields.

Bringing this together, we have that $\Im(\rho_{E_i,10}) = G_i$ for each $i\in\{1,2,3,4\}$ and the following diagrams:
\begin{figure}[h!]
\[
\begin{tikzcd}[
  column sep={6em,between origins},
  row sep={6em,between origins},
]
E_1 \arrow{r}{5-\text{isogenous}} \arrow{d}[swap]{\text{Twisting by }5} & E_2 \arrow{l}\arrow{d}{\text{Twisting by }5} & & G_1 \arrow{r}{5-\text{isogenous}} \arrow[equal]{d} & G_2 \arrow{l}\arrow[equal]{d} \\
E_4 \arrow{u} \arrow{r}{5-\text{isogenous}} & E_3 \arrow{l} \arrow{u} & &  G_4  \arrow{r}{5-\text{isogenous}} & G_3 \arrow{l} 
\end{tikzcd}
\]
\caption{Constellation of entanglements}
\label{fig:constellation}
\end{figure}
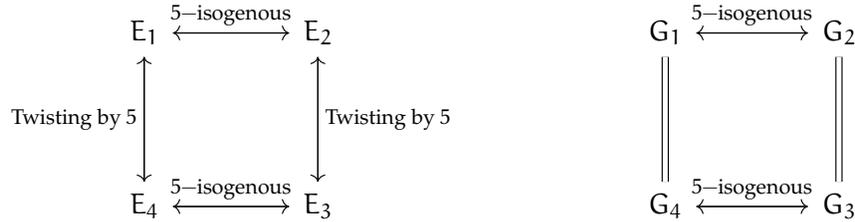
Each square commutes, and the two squares are connected by the map that takes $E_i$ to $\Im(\rho_{E_i,10})$. The vertical maps on the left are 5-twists and on the right are isomorphisms that connect the groups. The horizontal arrows on the left square are 5-isogenies. On the right, the horizontal arrows correspond to how $\Im(\rho_{E_1,10})$ (resp.~$\Im(\rho_{E_3,10})$) transforms to $\Im(\rho_{E_2,10})$ (resp.~$\Im(\rho_{E_4,10})$) under the $5$-isogeny $E_1\to E_2$ (resp.~$E_3 \to E_4$). We refer the reader to Lemma \ref{lem:isogrep} for a description of how these images transform.

As mentioned at the start of this subsection, we want to compute the Weierstrass equation for the universal family of non-CM elliptic curves over $\mQ$ where each member of the family has mod 10 image contained in $G_i$ for each $i = 1,2,3,4$. 
To do so, it suffices to determine the Weierstrass equation for the universal elliptic curve corresponding to ${G_1}$. 
Indeed, the above discussion tells us that we can use isogenies and twisting to determine the Weierstrass equation for the remaining families. 
By studying properties of $E_1$, we get a sense for how to compute the Weierstrass equation for the universal elliptic curve over ${G_1}$. 
The curve $E_1$ is just a twist of an elliptic curve with a point of order 5 by $-110$, which happens to be the square-free part of $\Delta_{E_1}$. Motivated by this, we use work of Zywina \cite{zywinapossible} to find a model for the generic elliptic curve with a point of order 5 
$$\mathcal{E}_t\colon y^2 = x^3 + (-27t^4 + 324t^3 - 378t^2 - 324t - 27)x + (54t^6 - 972t^5 + 4050t^4 + 4050t^2 + 972t + 54),$$
and note that its discriminant is $\Delta_{\mathcal{E}_t}\equiv t(t^2-11t-1) \bmod (\QQ(t)^\times)^2$. 
Let $d := t(t^2-11t-1)$.
Twisting $\mathcal{E}_t$ by $d$ we get 
\begin{align*}
\mathcal{E}_t^{(d)}\colon y^2 &= x^3-27t^2(t^2-11-1)^2(t^4-12t^3+14t^2+12t+1)x\\ & \phantom{ = }+ 54t^3(t^2-11-1)^3(t^4-12t^3+14t^2+12t+1)(t^2+1).
\end{align*}

By construction, this is the Weierstrass equation for the universal family of non-CM elliptic curves over $\mQ$ where each member of the family has mod 10 image contained in $G_1$, and a quick check shows the specialization of $\mathcal{E}_t^{(d)}$ to either $t\in\{10, -1/10 \}$ is exactly $E_1$. 
By our above discussion, we also have that $\mathcal{E}_t^{(5d)}$ is the Weierstrass equation for the universal family of non-CM elliptic curves over $\mQ$ where each member of the family has mod 10 image contained in $G_4$. 
Furthermore, the curve that is 5-isogenous to $\mathcal{E}_t^{(d)}$ is exactly the the Weierstrass equation for the universal family of non-CM elliptic curves over $\mQ$ where each member of the family has mod 10 image contained in $G_2$, and twisting this universal elliptic curve by 5, we get the Weierstrass equation for the universal family of non-CM elliptic curves over $\mQ$ where each member of the family has mod 10 image contained in $G_3$. 

To summarize, we are able to compute all 4 of the Weierstrass equations for the universal families using the constellation from Figure \ref{fig:constellation}.

\begin{remark}
Keeping the above notation, we have that $\langle G_1,-I \rangle = \langle G_4, -I\rangle$ and $\langle G_2,-I\rangle = \langle G_3,-I\rangle$, and so there are just two modular curves in the above discussion. 
We also note that if we intersect these groups with $\SL_2(\zZ{10})$, then they are all equal, and hence the two underlying modular curves are isomorphic over $\mQ(\zeta_5)$. 
\end{remark}

\begin{remark}\label{rem:restrict}
It turns out that this approach to computing modular curves of entanglement groups without $-I$ extends to cover all the examples that we are interested in. The example worked out here is the most complicated situation that we find when restricting to $(p,q)$-entanglement, but as soon as one considers $(p^n,q^m)$-entanglements the situation can become much more complicated. 
\end{remark}

\section{\bf Unexplained entanglements --- genus zero setting}
\label{sec:UnexplainedGenus0}
In this section, we describe two examples of constructing genus zero modular curves corresponding to an unexplained entanglement. These cases needed to be done by hand as the Siegel functions method did not work because we are not able to find a hauptmodul.

\begin{example}\label{exam:Nn3}
Recall that $3{\texttt{Nn}}$ is the normalizer of the nonsplit Cartan subgroup of $\GL_2(\ZZ/3\ZZ)$.
Our goal is to find the elliptic curves $E/\QQ$ with $\Im(\rho_{E,3})$ conjugate to a subgroup of $3{\texttt{Nn}}$ with the property that $\QQ(E[2]) \cap \QQ(E[3])$ is a quadratic extension different from $\QQ(\sqrt{-3})$.
If $E/\QQ$ is an elliptic curve with $\Im(\rho_{E,3}) \simeq 3{\texttt{Nn}}$, then $\QQ(E[3])$ has 3 different quadratic extensions, and hence there are only 2 ways to for $E$ to have an unexplained $(2,3)$-entanglement of type $\zZ{2}$.

By \cite{zywinapossible}, we know that elliptic curves over $\mQ$ such that $\Im(\rho_{E,3})$ is conjugate to a subgroup of $3{\texttt{Nn}}$ are exactly the elliptic curves with $j(E) = t^3$ for some $t\in\QQ$. 
Let $\mathcal{E}_t/\QQ(t)$ be the elliptic curve 
\[
\mathcal{E}_t\colon y^2 + xy = x^3 - \frac{36}{t^3 - 1728}x - \frac{1}{t^3 - 1728},
\]
which is the generic elliptic curve with $j$-invariant $t^3$.

Letting $f_3(x)$ be the 3-division polynomial of $\mathcal{E}_t$, we see that $\Gal(\QQ(t)(x(\mathcal{E}_t[3]))/\QQ(t))\simeq D_4$ so all of the quadratic extensions in $\QQ(t)(\mathcal{E}_t[3])$ are in fact contained in $\QQ(t)(x(\mathcal{E}_t[3]))$. Using \texttt{\texttt{Magma}}, we see that these quadratic extensions are exactly $$\QQ(t)\left(\sqrt{-3}\right),\ \QQ(t)\left( \sqrt{t^2 + 12t + 144} \right),\ \hbox{and }\QQ(t)\left( \sqrt{-3(t^2 + 12t + 144)} \right)$$
and that the unique quadratic extension in $\QQ(t)(\mathcal{E}_t[2])$ is exactly $\QQ(t)\left(\sqrt{\Delta_{\mathcal{E}_t}}\right)$. 
Again, using \texttt{\texttt{Magma}}, we see that $\QQ\left(\sqrt{\Delta_{\mathcal{E}_t}}\right) \simeq \QQ\left(\sqrt{ (t-12)(t^2+12t+144) }\right)$, and so 
in order to have an unexplained $(2,3)$-entanglement of type $\ZZ/2\ZZ$, we need to have 
\[(t-12)(t^2+12t+144) \equiv  \begin{cases}
\phantom{-2}(t^2 + 12t + 144) \bmod \left(\QQ(t)^\times\right)^2,\\
-3(t^2 + 12t + 144) \bmod \left(\QQ(t)^\times\right)^2.
 \end{cases}\]
Solving these equations, we see that these cases correspond to when 
\[
j(E) = 
\left\lbrace
\begin{array}{ll}
(s^2+12)^3 & \hbox{ for some }s\in\QQ,\\ 
\left(-\frac{s^2}{3}+12\right)^3 &   \hbox{ for some }s\in\QQ.
\end{array}
\right.
\]

Our search for admissible subgroups of $\GL_2(\ZZ/6\ZZ)$ that represent a nontrivial entanglement found 2 groups representing a $(2,3)$-entanglement of type $\ZZ/2\ZZ$  whose mod 3 image is conjugate to a subgroup of $3{\texttt{Nn}}$. Those groups are
$$G_1 = \left\langle 
\begin{pmatrix}
0 & 1\\
1 & 3
\end{pmatrix}
,
\begin{pmatrix}
3 & 5\\
1 & 0
\end{pmatrix}
,
\begin{pmatrix}
5 & 4\\
5 & 5
\end{pmatrix}
\right\rangle\hbox{\ \ and \ \ }
G_2 = \left\langle 
\begin{pmatrix}
5 & 1 \\
4 & 1
\end{pmatrix}
,
\begin{pmatrix}
5 & 1 \\
5 & 2
\end{pmatrix}
\right\rangle.$$

To distinguish which $j(E)$ corresponds to which $G_i$, we will evaluate each $j$-map at a ``generic point'' and see if we can determine which elliptic curve over $\QQ$ corresponds to which group. Choosing $s = 3$ and evaluating the above $j$-maps, we get that the curves of minimal conductor with the corresponding $j$-invariants are the curves $E_1/\QQ$ and $E_2/\QQ$ whose Cremona labels are \href{https://www.lmfdb.org/EllipticCurve/Q/25947d1/}{\texttt{25947d1}} and \href{https://www.lmfdb.org/EllipticCurve/Q/36963o1/}{\texttt{36963o1}}, respectively. 
According to LMFDB, these curves have full mod 2 image and mod 3 image exactly $\texttt{3Nn}$, and thus are generic. 

For these curves, we have 
\[
\begin{array}{lcr}
\QQ( \sqrt{ \Delta_{E_1} }) = \QQ(\sqrt{ 93 })& \hbox{ and } & \QQ(\sqrt{\Delta_{E_2}}) = \QQ(\sqrt{-111}),
\end{array}
\]
and 
\[
\begin{array}{lcr}
\QQ\left( \sqrt{ \Delta_{E_1} }\right) \subseteq \QQ\left(\sqrt{-3},\sqrt{-31})\right) \subseteq\QQ\left(x(E_1[3])\right)
& \hbox{ and } &
\QQ\left( \sqrt{ \Delta_{E_2} }\right) \subseteq \QQ\left(\sqrt{-3},\sqrt{37}\right) \subseteq\QQ\left(x(E_2[3])\right).
\end{array}
\]
Using \texttt{\texttt{Magma}} we confirm that, 
\[
\begin{array}{lcr}
L_1 = \QQ(E_1[2])\cap\QQ(E_1[3]) =\QQ(\sqrt{93})
& \hbox{ and } &
L_2 = \QQ(E_2[2])\cap\QQ(E_2[3]) =\QQ(\sqrt{-111}).
\end{array}
\]
Further, from the Galois correspondence, we know that $\Gal(\QQ(E_i[3])/L_i)\simeq \pi_3(N_{2,3}(G_i))$ for either $i=1$ or 2. 
Checking \texttt{\texttt{Magma}}, we see that $\pi_3(N_{2,3}(G_1))$ is not abelian while $\pi_3(N_{2,3}(G_2))$ is abelian, and one last computation shows that $\Gal(\QQ(E_1[3])/L_1)$ is not abelian, while $\Gal(\QQ(E_2[3])/L_2)$ is. 
From this we can see that: 
\begin{align*}
j_1\colon X_{G_1} \to X(1), & \quad  t \mapsto \left(-\frac{t^2}{3}+12\right)^3 \\
j_2\colon X_{G_2} \to X(1), & \quad t \mapsto  \left(t^2+12\right)^3 .
\end{align*}

\end{example}

\begin{remark}
Since the intersections here occur between $\QQ(E_i[2])$ and $\QQ\big( x(E_i[3]) \big)$ and both of these fields are invariant under twisting, the entanglements are independent of the choice of twist. So any twist of the curves $E_1$ and $E_2$ above will have the exact same entanglement, even up to the particular quadratic extension that is in both fields. One can see this on the group theoretic side by the fact that $-I\in G_i$ for $i=1$ or $2$. 
\end{remark}

\begin{remark}[Reality check]
For $G_1$ and $G_2$, we have that $G_1 \cap \SL_2(\zZ{6}) = G_2 \cap \SL_2(\zZ{6})$, and hence the modular curves $X_{G_1}$ and $X_{G_2}$ are isomorphic over $\mQ(\zeta_6)$. 
From the above $j$-maps, this is immediately clear. 
\end{remark}

\begin{example}\label{exam:Jeremy}
Another interesting group we encounter, which the Siegel functions method could not handle, comes from the admissible subgroup $G$ of $\GL_2(\zZ{10})$ with generators
\[
G := \left\langle \left(\begin{matrix}2 & 9 \\ 7 & 1\end{matrix}\right),\left(\begin{matrix}6 & 3 \\ 1 & 7\end{matrix}\right),\left(\begin{matrix}4 & 9 \\ 9 & 6\end{matrix}\right) \right\rangle.
\]
The group $G$ represents an unexplained $(2,5)$-entanglement of type $(S_3,\zZ{3})$. 
We know an elliptic curve with mod 10 image of Galois conjugate to a subgroup of $G$ has surjective mod 2 image and the mod 5 image is conjugate to a subgroup of the unique maximal subgroup $H$ of $\GL_2(\zZ{5})$ which contains the normalizer of the split Cartan at 5. In fact, the group $H$ is an 
{exceptional} group whose projective image is isomorphic to $S_4$. These kinds of groups can only occur for primes less than 11.
We know that $X_G$ is a degree 6 cover of $X_H$ and so if we could determine the map $X_G\to X_H$ then we could determine the $j$-map for $X_G$ since work of Zywina \cite{zywinapossible} tells us that the $j$-map $X_H\to X(1)$ is given by $ j = x^5 + 5x^4 + 40x^3$. 

The subgroup $G$ has a unique index $3$ subgroup. 
This index $3$ subgroup must correspond to \textit{both} the field over which the elliptic curve on $X_G$ obtains a 2-torsion point and the field over which you obtain a rational point on the modular curve $X_s^+(5)$ associated to the normalizer of the split Cartan at 5. 
Now, an elliptic curve with mod 5 image equal to $H$ will have a degree 3 extension over which it attains a point of order 2 and another degree 3 extension over which it corresponds to a point on $X_s^+(5)$, and we want to know when these degree 3 extensions coincide.

First, we use the equation $t^3 - jt - 16j = 0$ for $X_0(2)$, which was computed in \cite{rouse2014elliptic}. 
Second, recent work of Rouse--Sutherland--Zureick-Brown \cite{RouseSZB:elladic} determines that the morphism $X_s^+(5) \to X_H$ is given by $x = ((5/3)t^3 - 8t - 8/3)/(t^2 + t - 1/5)$.  
Using the methods from \cite[p.~19]{bourdonGRW:Isolatedpoints} and \cite{vanHoeijN:AlgebraicFunctionField}, we can produce the degree $6$ cover of $X_H$, which parametrizes when the above two degree 3 extensions coincide. 
Moreover, this gives us our desired degree 6 cover $X_G \to X_H$, which is defined by 
\[x = (8y^6 + 8y^5 - 20y^4 - 50y^3 + 80y^2 - 12y + 3)/((y+1)^2(y^2-3y+1)^2).\] 
By composing the morphisms $X_G \to X_H \to X(1)$, we have our desired $j$-map. 
\end{example}

\begin{remark}
We note that our construction of the modular curve $X_G$ in Example \ref{exam:Jeremy} agrees with the construction from \cite[Subsection 4.3]{JonesM:Nonabelianentanglements}.
\end{remark}

\section{\bf Unexplained entanglements --- genus one setting}
\label{sec:UnexplainedGenus1}
\label{sec:genusone}
In our search, we find that there are exactly 2 maximal groups of genus 1 with positive rank that represent unexplained entanglements; in fact, both represent $(2,7)$-entanglements of type $\ZZ/2\ZZ$ and both groups have surjective image under $\pi_2$. 
The difference between these groups is their image under $\pi_7$:~one of them has image $7{\texttt{Ns}}$ and the other has image $7{\texttt{Nn}}$. 
We let these groups of level 14 be $G_s$ and $G_n$ respectively and provide generators for them below
\begin{align*}
G_s &=\left\langle
\begin{pmatrix}
3& 2\\ 5& 11
\end{pmatrix},
\begin{pmatrix}
2& 13\\ 1& 12
\end{pmatrix},
\begin{pmatrix}
0& 9\\ 5& 0 
\end{pmatrix},
\begin{pmatrix}
13& 12\\ 9& 1 
\end{pmatrix}
\right\rangle\\
G_n&=\left\langle
\begin{pmatrix}
12& 11\\ 11& 3
\end{pmatrix},
\begin{pmatrix}
0& 9\\ 9& 13
\end{pmatrix},
\begin{pmatrix}
8& 9\\ 1& 6
\end{pmatrix}
\right\rangle.
\end{align*}
At first these groups seem quite different, but we will prove that their modular curves are isomorphic over $\mQ$. 

\begin{proposition}[= Proposition \ref{prop:iso}]
The genus one modular curves of positive rank $X_{G_s}$ and $X_{G_n}$ are both isomorphic to the elliptic curve \href{https://www.lmfdb.org/EllipticCurve/Q/196a1/}{\emph{\texttt{196a1}}}. 
\end{proposition}

\subsection{Computing $X_{G_s}$}
Let $\mathcal{E}_t/\mQ(t)$ be the generic elliptic curve with mod 7 image of Galois contained in $7{\texttt{Ns}}$ twisted to minimize its conductors and discriminant. 
We do this by constructing the elliptic curve with $j$-invariant equal to the $j$-map $j_s\colon X_s^+(7)\to X(1)$ using \cite{zywinapossible} and then twisting. 
We compute that
\[
\Delta_{\mathcal{E}_t} \equiv (t^3 - 4t^2 + 3t + 1)(t^4 - 10t^3 + 27t^2 - 10t - 27) \mod (\QQ(t)^\times)^2
\]
and the 7-division polynomial of $\mathcal{E}_t$, call it $f_{7}(x)$, factors into a degree 6 polynomial times a degree 18 polynomial. 
Further, we see that the degree 6 polynomial factors over the field 
\[
K_s = \QQ(t)\left(\sqrt{t^4 - 10t^3 + 27t^2 - 10t - 27}\right), 
\]
and now, we notice that $\QQ(\sqrt{\Delta_{\mathcal{E}_t}}) = K_s$ exactly when $(t^3 - 4t^2 + 3t + 1)$ is a square. 

As we are considering unexplained $(2,7)$-entanglements of type $\zZ{2}$, we see that the modular curve $X_{G_s}$ is isomorphic to the elliptic curve with Weierstrass equation $ y^2 = (x^3 - 4x^2 + 3x + 1)$, which is the rank 1 elliptic curve \href{https://www.lmfdb.org/EllipticCurve/Q/196a1/}{\texttt{196a1}}. 
Using \texttt{\texttt{Magma}}, we find that 
\[X_{G_s}(\QQ) = \langle (1,1),(0,0) \rangle \simeq \ZZ \oplus \ZZ/2\ZZ\]
and the $j$-map $X_{G_s} \to X(1)$ is exactly the map $P = (x,y) \mapsto j_s(x)$ where 
$$j_s(t) = \frac{t(t+1)^3(t^2-5t+1)^3(t^2-5t+8)^3(t^4 - 5t^3 + 8t^2 - 7t + 7)^2}{(t^3 - 4t^2 + 3t + 1)^7},$$ 
is the aforementioned  $j$-map of $X_s^+(7)$ taken from \cite{zywinapossible}.

\subsection{Computing $X_{G_n}$}
We proceed the same way as before. This time we let  $\mathcal{E}_t$ be the generic elliptic curve with mod 7 image of Galois contained in $7{\texttt{Nn}}$ twisted to minimize its conductor and discriminant. In this case, 
$$\Delta_{\mathcal{E}_t} \equiv (t-3)(t^3-7t^2+7t+7)(2t^4 - 14t^3 + 21t^2 + 28t + 7) \mod (\QQ(t)^\times)^2$$
If we let $d_2 = 7(2t^4 - 14t^3 + 21t^2 + 28t + 7)$ we can see that the 7-division polynomial of $\mathcal{E}_t$, call it $f_7(x)$, is irreducible over $\QQ(t)$ but factors over $\QQ(t)(\sqrt{d_2})$. Thus $\QQ(t)(\sqrt{d_2}) \subseteq \QQ(t)(\mathcal{E}_t[7]).$ 
Since $\QQ(t)(\zeta_7)\subseteq\QQ(t)(\mathcal{E}_t[7])$, we also have that $\QQ(t)(\sqrt{-7d_2}) \subseteq \QQ(t)(\mathcal{E}_t[7])$. 
A simple inspection shows that $\QQ(t)(\sqrt{\Delta_{\mathcal{E}_t}}) = \QQ(t)(\sqrt{-7d_2})$ exactly when $-(t-3)(t^3-7t^2+7t+7)$ is a square. 

Therefore, we are looking for values of $t\in\QQ$ such that there is a $y\in \QQ$ with $y^2 = -(t-3)(t^3-7t^2+7t+7)$. This again defines out a genus 1 curve and miraculously it is $\QQ$-isomorphic to \href{https://www.lmfdb.org/EllipticCurve/Q/196a1/}{\texttt{196a1}}. That is to say, $X_{G_s} \simeq_\QQ X_{G_n}$, and the only difference is their j-maps. 

Let $E\colon y^2 = x^3 - 4x^2 + 3x + 1 $, let $C/\QQ$ be the curve
$$C \colon y^2 = -(x-3)(x^3-7x^2+7x+7),$$ 
and let $\psi\colon E \simeq C$ be an isomorphism between them. 
We have that $X_{G_n} \simeq E$ and the $j$-map of $X_{G_n}$ is define by taking a point $P\in C(\mQ)$ and mapping it to $j_n(\psi^{-1}(P)_x)$ where $\psi^{-1}(P)_x$ is the $x$-coordinate of $\psi^{-1}(P)$ and $j_n\colon X_n^+(7)\to X(1)$ is the $j$-map from Zywina \cite{zywinapossible} given by
$$j_n(t) = 1600\frac{t^3(t^2-7t+14)^3(5t^2-14t-7)^3(t^2+7)^3}{(t^3 - 7t^2 + 7t + 7)^7}.$$

\section{\bf Appendix --- Tables} 
\label{sec:tables}
To conclude, we provide tables of the various modular curves we constructed whose rational points parametrize elliptic curves with an unexplained $(p,q)$-entanglement of type $T$ for the pairs $((p,q),T)$ mentioned in Theorem \ref{thm:main}.

There are three subsections:~genus 0 groups with $-I$, genus 0 groups without $-I$, and genus 1 groups (with $-I$). 
In each subsection, there is a table for each pair $((p,q),T)$ where $(p,q)$ is the level of the entanglement and $T$ is the type of entanglement. 
In each table, we will provide the label (see our group conventions) for the group $G$ representing an unexplained $(p,q)$-entanglement of type $T$, the generators for the group $G$, a parametrization for the modular curve $X_G$, and finally an example from LMFDB of an elliptic curve with an unexplained $(p,q)$-entanglement of type $T$. 
These examples were chosen to minimize the conductor, and when possible, we tried to find non-CM examples. 
When a CM example is presented, this means that there were no non-CM examples in the entire LMFDB database which had this prescribed composite level image of Galois.

\subsection{Genus 0 groups with $-I$}
In these tables, the parametrization of the modular curve comes as a rational function $f(t)$. 

\begin{center}
\renewcommand{\arraystretch}{1.5}
\begin{tabular}{| c | c | c | c |}\hline
\multicolumn{4}{| c |}{ \textbf{$(2,3)$-entanglements of type $\zZ{2}$}}   \\\hline
\textsc{Label} & \textsc{Generators} & \textsc{$j$-map} & \textsc{Example} \\\hline\hline
$[ \texttt{GL2}, \texttt{3Ns} ]$ & $\left\langle \left(\begin{smallmatrix}3 & 5 \\ 4 & 3\end{smallmatrix}\right),\left(\begin{smallmatrix}5 & 3 \\ 3 & 2\end{smallmatrix}\right),\left(\begin{smallmatrix}2 & 3 \\ 3 & 1\end{smallmatrix}\right) \right\rangle$ & $ \frac{(t - 3)^{3}(t + 3)^{3}(t^{2} + 3)^{3}}{t^{6}}$ & \href{https://www.lmfdb.org/EllipticCurve/Q/6627e1/}{\texttt{6627e1}} \\ 
$[ \texttt{GL2}, \texttt{3Nn} ]$ & $\left\langle \left(\begin{smallmatrix}5 & 1 \\ 4 & 1\end{smallmatrix}\right),\left(\begin{smallmatrix}5 & 1 \\ 5 & 2\end{smallmatrix}\right) \right\rangle$ & $(t^{2} + 12)^{3}$ & \href{https://www.lmfdb.org/EllipticCurve/Q/1369e1/}{\texttt{1369e1}} \\ 
$[ \texttt{GL2}, \texttt{3Nn} ]$ & $\left\langle \left(\begin{smallmatrix}0 & 1 \\ 1 & 3\end{smallmatrix}\right),\left(\begin{smallmatrix}3 & 5 \\ 1 & 0\end{smallmatrix}\right),\left(\begin{smallmatrix}5 & 4 \\ 5 & 5\end{smallmatrix}\right) \right\rangle$ & $3^{3}(t - 2)^{3}(t + 2)^{3}$ & \href{https://www.lmfdb.org/EllipticCurve/Q/31046b2/}{\texttt{31046b2}} \\ 
\hline 
\end{tabular}
\end{center}
\vspace{7pt}

\begin{center}
\renewcommand{\arraystretch}{1.5}
\begin{tabular}{| c | c | c | c |}\hline
\multicolumn{4}{| c |}{\textbf{ $(2,3)$-entanglements of type $(S_3,\zZ{3})$}}   \\\hline
\textsc{Label} & \textsc{Generators} & \textsc{$j$-map} & \textsc{Example} \\\hline\hline
$[ \texttt{GL2}, \texttt{GL3} ]$ & $\left\langle \left(\begin{smallmatrix}5 & 1 \\ 4 & 3\end{smallmatrix}\right),\left(\begin{smallmatrix}4 & 1 \\ 1 & 0\end{smallmatrix}\right) \right\rangle$ & $2^{10}3^{3}t^{3}(4t^{3} - 1)$ & \href{https://www.lmfdb.org/EllipticCurve/Q/300a1/}{\texttt{300a1}} \\ 
\hline \end{tabular}
\end{center}
\vspace{7pt}

\begin{center}
\renewcommand{\arraystretch}{1.5}
\begin{tabular}{| c | c | c | c |}\hline
\multicolumn{4}{| c |}{ \textbf{$(2,5)$-entanglements of type $\zZ{2}$}}   \\\hline
\textsc{Label} & \textsc{Generators} & \textsc{$j$-map} & \textsc{Example} \\\hline\hline
$[ \texttt{GL2}, \texttt{5B} ]$ & $\left\langle \left(\begin{smallmatrix}9 & 5 \\ 1 & 2\end{smallmatrix}\right),\left(\begin{smallmatrix}4 & 5 \\ 1 & 9\end{smallmatrix}\right),\left(\begin{smallmatrix}8 & 5 \\ 1 & 8\end{smallmatrix}\right) \right\rangle$ & $ \frac{(t^{4} + 10t^{2} + 5)^{3}}{t^{2}}$ & \href{https://www.lmfdb.org/EllipticCurve/Q/1369e1/}{\texttt{1369e1}} \\ 
$[ \texttt{GL2}, \texttt{5B} ]$ & $\left\langle \left(\begin{smallmatrix}9 & 5 \\ 9 & 4\end{smallmatrix}\right),\left(\begin{smallmatrix}7 & 5 \\ 9 & 4\end{smallmatrix}\right),\left(\begin{smallmatrix}2 & 5 \\ 7 & 2\end{smallmatrix}\right) \right\rangle$ & $ \frac{(t^{4} + 50t^{2} + 125)^{3}}{5^5t^{2}}$ & \href{https://www.lmfdb.org/EllipticCurve/Q/1369e2/}{\texttt{1369e2}} \\ 
$[ \texttt{GL2}, \texttt{5Nn} ]$ & $\left\langle \left(\begin{smallmatrix}7 & 7 \\ 9 & 8\end{smallmatrix}\right),\left(\begin{smallmatrix}4 & 3 \\ 3 & 8\end{smallmatrix}\right) \right\rangle$ & $\frac{2^{15}5^4 t^{3}(20t^{2} - 20t + 1)(400t^{4} + 200t^{3} + 80t^{2} + 10t + 1)^{3}}{(20t^{2} - 1)^{10}}$ & \href{https://www.lmfdb.org/EllipticCurve/Q/4900l1/}{\texttt{4900l1}} \\ 
$[ \texttt{GL2}, \texttt{5Nn} ]$ & $\left\langle \left(\begin{smallmatrix}7 & 2 \\ 9 & 3\end{smallmatrix}\right),\left(\begin{smallmatrix}9 & 3 \\ 3 & 8\end{smallmatrix}\right),\left(\begin{smallmatrix}4 & 3 \\ 3 & 3\end{smallmatrix}\right) \right\rangle$ & $\frac{(-5)^3 (t - 1)(5t - 1)(5t^{2} - 10t + 1)^{3}(5t^{2} + 3)^{3}(15t^{2} + 1)^{3}}{(5t^{2} - 1)^{10}}$ & \href{https://www.lmfdb.org/EllipticCurve/Q/27a1/}{\texttt{27a1}} \\ 
\hline \end{tabular}
\end{center}
\vspace{7pt}

\begin{center}
\renewcommand{\arraystretch}{1.5}
\begin{tabular}{| c | c | c | c |}\hline
\multicolumn{4}{| c |}{ \textbf{$(2,5)$-entanglements of type $(S_3,\zZ{3})$}}   \\\hline
\textsc{Label} & \textsc{Generators} & \textsc{$j$-map} & \textsc{Example} \\\hline\hline
$[ \texttt{GL2}, \texttt{5S4} ]$ & $\left\langle \left(\begin{smallmatrix}2 & 9 \\ 7 & 1\end{smallmatrix}\right),\left(\begin{smallmatrix}6 & 3 \\ 1 & 7\end{smallmatrix}\right),\left(\begin{smallmatrix}4 & 9 \\ 9 & 6\end{smallmatrix}\right) \right\rangle$ & $ \frac{(2t^{2} - t + 2)^{3}P_1(t)^{3}(18t^{6} - 12t^{5} - 70t^{4} + 25t^{3} + 130t^{2} - 52t + 8)}{(t + 1)^{10}(t^{2} - 3t + 1)^{10}}$ & \href{https://www.lmfdb.org/EllipticCurve/Q/3240a1/}{\texttt{3240a1} }\\ 
\hline \end{tabular}
\end{center}
\begin{align*}
P_1(t) &= (8t^{6} + 8t^{5} - 20t^{4} - 50t^{3} + 80t^{2} - 12t + 3).
\end{align*}

\vspace{7pt}

\begin{center}
\setlength{\extrarowheight}{1.4pt}
\begin{tabular}{| c | c | c | c |}\hline
\multicolumn{4}{| c |}{ \textbf{$(2,7)$-entanglements of type $\zZ{3}$}}   \\\hline
\textsc{Label} & \textsc{Generators} & \textsc{$j$-map} & \textsc{Example} \\\hline\hline
$[ \texttt{2Cn}, \texttt{7B} ]$ & $\left\langle \left(\begin{smallmatrix}10 & 7 \\ 5 & 11\end{smallmatrix}\right),\left(\begin{smallmatrix}2 & 7 \\ 3 & 3\end{smallmatrix}\right),\left(\begin{smallmatrix}5 & 7 \\ 9 & 12\end{smallmatrix}\right) \right\rangle$& $ \frac{(t^{2} + t + 1)^{3}(t^{6} + 5t^{5} + 12t^{4} + 9t^{3} + 2t^{2} + t + 1)P_2(t)^{3}}{t^{14}(t + 1)^{14}(t^{3} + 2t^{2} - t - 1)^{2}}$ & \href{https://www.lmfdb.org/EllipticCurve/Q/1922e1/}{\texttt{1922e1}} \\ 
$[ \texttt{2Cn}, \texttt{7B} ]$ & $\left\langle \left(\begin{smallmatrix}5 & 0 \\ 8 & 5\end{smallmatrix}\right),\left(\begin{smallmatrix}3 & 0 \\ 8 & 11\end{smallmatrix}\right),\left(\begin{smallmatrix}9 & 7 \\ 7 & 10\end{smallmatrix}\right) \right\rangle$ & $ \frac{7^{4}(t^{2} + t + 1)^{3}(9t^{6} + 39t^{5} + 64t^{4} + 23t^{3} + 4t^{2} + 15t + 9)P_3(t)^{3}}{(t^{3} + t^{2} - 2t - 1)^{14}(t^{3} + 8t^{2} + 5t - 1)^{2}}$ & \href{https://www.lmfdb.org/EllipticCurve/Q/3969c2/}{\texttt{3969c2}} \\ 
$[ \texttt{2Cn}, \texttt{7B} ]$ &  $\left\langle \left(\begin{smallmatrix}9 & 7 \\ 3 & 10\end{smallmatrix}\right),\left(\begin{smallmatrix}5 & 7 \\ 9 & 12\end{smallmatrix}\right) \right\rangle$ & $ \frac{(t^{2} - t + 1)^{3}(t^{6} - 5t^{5} + 12t^{4} - 9t^{3} + 2t^{2} - t + 1)P_4(t)^{3}}{(t - 1)^{2}t^{2}(t^{3} - 2t^{2} - t + 1)^{14}}$ & \href{https://www.lmfdb.org/EllipticCurve/Q/1922e2/}{\texttt{1922e2}} \\ 
\hline \end{tabular}
\end{center}
\begin{align*}
P_2(t) &= t^{12} + 8t^{11} + 25t^{10} + 34t^{9} + 6t^{8} - 30t^{7} - 17t^{6} + 6t^{5} - 4t^{3} + 3t^{2} + 4t + 1,\\
P_3(t) &= t^{12} + 18t^{11} + 131t^{10} + 480t^{9} + 1032t^{8} + 1242t^{7} + 805t^{6} + 306t^{5} + 132t^{4} + 60t^{3} - t^{2} - 6t + 1,\\
P_4(t) &= t^{12} - 8t^{11} + 265t^{10} - 1474t^{9} + 5046t^{8} - 10050t^{7} + 11263t^{6} - 7206t^{5} + 2880t^{4} - 956t^{3}\\ &\phantom{= t^{12}}+ 243t^{2} - 4t + 1.
\end{align*}

\begin{center}
\renewcommand{\arraystretch}{1.5}
\begin{tabular}{| c | c | c | c |}\hline
\multicolumn{4}{| c |}{ \textbf{$(2,13)$-entanglements of type $\zZ{2}$}}   \\\hline
\textsc{Label} & \textsc{Generators} & \textsc{$j$-map} & \textsc{Example} \\\hline\hline
$[ \texttt{GL2}, \texttt{13B} ]$ & $\left\langle \left(\begin{smallmatrix}21 & 6 \\ 21 & 17\end{smallmatrix}\right),\left(\begin{smallmatrix}5 & 23 \\ 22 & 7\end{smallmatrix}\right),\left(\begin{smallmatrix}12 & 15 \\ 25 & 17\end{smallmatrix}\right),\left(\begin{smallmatrix}11 & 11 \\ 12 & 13\end{smallmatrix}\right) \right\rangle$ & $\frac{(13t^{4} + 5t^{2} + 1)(28561t^{8} + 15379t^{6} + 3380t^{4} + 247t^{2} + 1)^{3}}{7t^{2}}$ & \href{https://www.lmfdb.org/EllipticCurve/Q/9025j2/}{\texttt{9025j2}}\\ 
$[ \texttt{GL2}, \texttt{13B} ]$ & $\left\langle \left(\begin{smallmatrix}5 & 18 \\ 12 & 9\end{smallmatrix}\right),\left(\begin{smallmatrix}17 & 20 \\ 5 & 21\end{smallmatrix}\right),\left(\begin{smallmatrix}14 & 21 \\ 15 & 0\end{smallmatrix}\right) \right\rangle$ & $ \frac{(t^{4} + 5t^{2} + 13)(t^{8} + 7t^{6} + 20t^{4} + 19t^{2} + 1)^{3}}{t^{2}}$ & \href{https://www.lmfdb.org/EllipticCurve/Q/9025j1/}{\texttt{9025j1}} \\ 
\hline \end{tabular}
\end{center}
\vspace{7pt}
\begin{center}
\renewcommand{\arraystretch}{1.5}
\begin{tabular}{| c | c | c | c |}\hline
\multicolumn{4}{| c |}{ \textbf{$(3,5)$-entanglements of type $\zZ{2}$}}   \\\hline
\textsc{Label} & \textsc{Generators} & \textsc{$j$-map} & \textsc{Example} \\\hline\hline
$[ \texttt{3Nn}, \texttt{5B} ]$ & $\left\langle \left(\begin{smallmatrix}10 & 7 \\ 8 & 4\end{smallmatrix}\right),\left(\begin{smallmatrix}10 & 8 \\ 4 & 4\end{smallmatrix}\right),\left(\begin{smallmatrix}11 & 2 \\ 8 & 10\end{smallmatrix}\right),\left(\begin{smallmatrix}13 & 5 \\ 1 & 7\end{smallmatrix}\right),\left(\begin{smallmatrix}1 & 11 \\ 4 & 13\end{smallmatrix}\right) \right\rangle$ & $ \frac{2^{12}P_5(t)^{3}}{(t - 1)^{15}(t + 1)^{15}(t^{2} - 4t - 1)^{3}}$ & \href{https://www.lmfdb.org/EllipticCurve/Q/1369e1/}{\texttt{1369e1}} \\ 
$[ \texttt{3Nn}, \texttt{5B} ]$ & $\left\langle \left(\begin{smallmatrix}1 & 2 \\ 13 & 10\end{smallmatrix}\right),\left(\begin{smallmatrix}14 & 5 \\ 7 & 2\end{smallmatrix}\right),\left(\begin{smallmatrix}11 & 1 \\ 14 & 8\end{smallmatrix}\right),\left(\begin{smallmatrix}11 & 11 \\ 13 & 14\end{smallmatrix}\right),\left(\begin{smallmatrix}10 & 8 \\ 2 & 11\end{smallmatrix}\right) \right\rangle$ & $ \frac{2^{12}P_6(t)^{3}}{(t - 1)^{3}(t + 1)^{3}(t^{2} - 4t - 1)^{15}}$ & \href{https://www.lmfdb.org/EllipticCurve/Q/1369e2/}{\texttt{1369e2}} \\ 
\hline \end{tabular}
\begin{align*}
P_5(t) &= t^{12} - 9t^{11} + 39t^{10} - 75t^{9} + 75t^{8} - 114t^{7} + 26t^{6} + 114t^{5} + 75t^{4} + 75t^{3} + 39t^{2} + 9t + 1,\\
P_6(t) &= 211t^{12} - 189t^{11} - 501t^{10} - 135t^{9} + 345t^{8} + 966t^{7} + 146t^{6} - 966t^{5} + 345t^{4} + 135t^{3}\\
&\phantom{= 211t^{12} } - 501t^{2} + 189t + 211.
\end{align*}
\end{center}

\begin{remark}
In the above tables, we note that the elliptic curves \href{https://www.lmfdb.org/EllipticCurve/Q/1369e1/}{\texttt{1369e1}} and \href{https://www.lmfdb.org/EllipticCurve/Q/1369e2/}{\texttt{1369e2}} both have an unexplained $(2,5)$-entanglement of type $\zZ{2}$ and an unexplained $(3,5)$-entanglement of type $\zZ{2}$. 
In fact, one can show that for these curves $\mQ(E[2])\cap \mQ(E[3])\cap \mQ(E[5]) \simeq \mQ(\sqrt{37})$, and so they have a $(2,3,5)$-entanglement of type $\zZ{2}$. 
While it would be interesting to classify elliptic curves with this entanglement, we do not attempt to do so here.
\end{remark}

\subsection{Genus 0 groups without $-I$}
In these tables, the parametrization of the modular curve comes in the form $a$-invariants for the generic elliptic curve with a prescribed image of Galois. 

\begin{center}
\renewcommand{\arraystretch}{1.5}
\begin{tabular}{| c | c | c | c |}\hline
\multicolumn{4}{| c |}{ \textbf{$(2,3)$-entanglements of type $\zZ{2}$}}   \\\hline
\textsc{Label} & \textsc{Generators} & \textsc{$a$-invariants} & \textsc{Example} \\\hline\hline
$[ \texttt{GL2}, \texttt{3B} ]$ & $\left\langle \left(\begin{smallmatrix}5 & 5 \\ 0 & 5\end{smallmatrix}\right),\left(\begin{smallmatrix}2 & 5 \\ 3 & 2\end{smallmatrix}\right),\left(\begin{smallmatrix}2 & 1 \\ 3 & 1\end{smallmatrix}\right) \right\rangle$ &  $[ Q_1(t), R_1(t) ]$ & \href{https://www.lmfdb.org/EllipticCurve/Q/73926l2/}{\texttt{73926l2}}\\ 
$[ \texttt{GL2}, \texttt{3B} ]$ & $\left\langle \left(\begin{smallmatrix}2 & 5 \\ 3 & 2\end{smallmatrix}\right),\left(\begin{smallmatrix}1 & 3 \\ 3 & 2\end{smallmatrix}\right) \right\rangle$ & $\left[  \frac{1}{3^2}Q_1(t), \frac{1}{3^3}R_1(t) \right]$ & \href{https://www.lmfdb.org/EllipticCurve/Q/73926x1/}{\texttt{73926x1}}\\ 
\hline \end{tabular}
\end{center}
\begin{align*}
Q_1(t) &= 3^{5}t^{2}(t + 1)(t + 3)(t^{2} + 3)(t^{2} + 3t + 3)^{2},\\
R_1(t) &= 23^{6}t^{3}(t^{2} - 3)(t^{2} + 3t + 3)^{3}(t^{4} + 6t^{3} + 18t^{2} + 18t + 9).
\end{align*}

\vspace{7pt}

\begin{center}
\renewcommand{\arraystretch}{1.5}
\begin{tabular}{| c | c | c | c |}\hline
\multicolumn{4}{| c |}{ \textbf{$(2,5)$-entanglements of type $\zZ{2}$}}   \\\hline
\textsc{Label} & \textsc{Generators} & \textsc{$a$-invariants} & \textsc{Example} \\\hline\hline
$[ \texttt{GL2}, \texttt{5B.4.2} ]$ & $\left\langle \left(\begin{smallmatrix}6 & 5 \\ 3 & 1\end{smallmatrix}\right),\left(\begin{smallmatrix}6 & 5 \\ 7 & 3\end{smallmatrix}\right),\left(\begin{smallmatrix}9 & 0 \\ 3 & 9\end{smallmatrix}\right) \right\rangle$ & $[  Q_2(t), R_2(t) ]$ & \href{https://www.lmfdb.org/EllipticCurve/Q/371522f1/}{\texttt{371522f1}}\\ 
$[ \texttt{GL2}, \texttt{5B.4.2} ]$ & $\left\langle \left(\begin{smallmatrix}9 & 5 \\ 5 & 8\end{smallmatrix}\right),\left(\begin{smallmatrix}6 & 5 \\ 1 & 1\end{smallmatrix}\right),\left(\begin{smallmatrix}9 & 0 \\ 3 & 9\end{smallmatrix}\right) \right\rangle$ & $[  5^{2}Q_2(t), 5^{3}R_2(t) ]$ & \href{https://www.lmfdb.org/EllipticCurve/Q/1225j2/}{\texttt{1225j2}} \\ 
$[ \texttt{GL2}, \texttt{5B.4.1} ]$ & $\left\langle \left(\begin{smallmatrix}6 & 5 \\ 3 & 1\end{smallmatrix}\right),\left(\begin{smallmatrix}3 & 5 \\ 9 & 6\end{smallmatrix}\right),\left(\begin{smallmatrix}9 & 0 \\ 3 & 9\end{smallmatrix}\right) \right\rangle$ & $[ Q_3(t) , R_3(t) ]$ & \href{https://www.lmfdb.org/EllipticCurve/Q/371522f2/}{\texttt{371522f2}}\\ 
$[ \texttt{GL2}, \texttt{5B.4.1} ]$ & $\left\langle \left(\begin{smallmatrix}6 & 5 \\ 1 & 1\end{smallmatrix}\right),\left(\begin{smallmatrix}2 & 5 \\ 5 & 9\end{smallmatrix}\right),\left(\begin{smallmatrix}9 & 0 \\ 3 & 9\end{smallmatrix}\right) \right\rangle$ & $[  5^{2}R_3(t), 5^{3}Q_3(t) ]$ & \href{https://www.lmfdb.org/EllipticCurve/Q/1225j2/}{\texttt{1225j2}}\\ 
\hline \end{tabular}
\end{center}
\begin{align*}
Q_2(t) &= 3^{3}t^{2}(t^{2} - 11t - 1)^{2}(t^{4} - 12t^{3} + 14t^{2} + 12t + 1),\\
R_2(t) &= 23^{3}t^{3}(t^{2} - 11t - 1)^{3}(t^{2} + 1)(t^{4} - 18t^{3} + 74t^{2} + 18t + 1),\\
Q_3(t) &= 3^{3}t^{2}(t^{2} - 11t - 1)^{2}(t^{4} + 228t^{3} + 494t^{2} - 228t + 1),\\
R_3(t) &= 23^{3}t^{3}(t^{2} - 11t - 1)^{3}(t^{2} + 1)(t^{4} - 522t^{3} - 10006t^{2} + 522t + 1).
\end{align*}

\begin{center}
\renewcommand{\arraystretch}{1.5}
\begin{tabular}{| c | c | c | c |}\hline
\multicolumn{4}{| c |}{ \textbf{$(2,7)$-entanglements of type $\zZ{2}$}}   \\\hline
\textsc{Label} & \textsc{Generators} & \textsc{$a$-invariants} & \textsc{Example} \\\hline\hline
$[ \texttt{GL2}, \texttt{7B} ]$ & $\left\langle \left(\begin{smallmatrix}11 & 7 \\ 1 & 10\end{smallmatrix}\right),\left(\begin{smallmatrix}10 & 7 \\ 1 & 10\end{smallmatrix}\right),\left(\begin{smallmatrix}3 & 0 \\ 5 & 3\end{smallmatrix}\right),\left(\begin{smallmatrix}5 & 7 \\ 6 & 11\end{smallmatrix}\right) \right\rangle$ & $[ Q_4(t), R_4(t) ]$ &\href{https://www.lmfdb.org/EllipticCurve/Q/19600db2/}{\texttt{19600db2}}  \\ 
$[ \texttt{GL2}, \texttt{7B} ]$ & $\left\langle \left(\begin{smallmatrix}5 & 7 \\ 8 & 5\end{smallmatrix}\right),\left(\begin{smallmatrix}10 & 7 \\ 5 & 1\end{smallmatrix}\right),\left(\begin{smallmatrix}12 & 7 \\ 3 & 12\end{smallmatrix}\right) \right\rangle$ & $\left[  \frac{1}{7^2}Q_4(t), \frac{1}{7^3}R_4(t) \right]$ & \href{https://www.lmfdb.org/EllipticCurve/Q/19600by2/}{\texttt{19600by2}}\\ 
\hline \end{tabular}
\end{center}
\begin{align*}
Q_4(t) &=  3^{3}7^{2}t^{2}(t^{2} + 13t + 49)^{3}(t^{2} + 245t + 2401),\\
R_4(t) &= 23^{3}7^{3}t^{3}(t^{2} + 13t + 49)^{4}(t^{4} - 490t^{3} - 21609t^{2} - 235298t - 823543). 
\end{align*}

\vspace{7pt}

\begin{center}
\renewcommand{\arraystretch}{1.5}
\begin{tabular}{| c | c | c | c |}\hline
\multicolumn{4}{| c |}{ \textbf{$(2,13)$-entanglements of type $\zZ{2}$}}   \\\hline
\textsc{Label} & \textsc{Generators} & \textsc{$a$-invariants} & \textsc{Example} \\\hline\hline
$[ \texttt{GL2}, \texttt{13B.4.1} ]$ & $\left\langle \left(\begin{smallmatrix}8 & 7 \\ 5 & 12\end{smallmatrix}\right),\left(\begin{smallmatrix}20 & 25 \\ 21 & 14\end{smallmatrix}\right),\left(\begin{smallmatrix}5 & 3 \\ 4 & 3\end{smallmatrix}\right) \right\rangle$ & $[ Q_5(t), R_5(t) ]$ & \href{https://www.lmfdb.org/EllipticCurve/Q/74529q1/}{\texttt{74529q1}}\\ 
$[ \texttt{GL2}, \texttt{13B.4.1} ]$ & $\left\langle \left(\begin{smallmatrix}0 & 23 \\ 21 & 25\end{smallmatrix}\right),\left(\begin{smallmatrix}15 & 15 \\ 20 & 5\end{smallmatrix}\right),\left(\begin{smallmatrix}25 & 24 \\ 19 & 9\end{smallmatrix}\right),\left(\begin{smallmatrix}5 & 16 \\ 9 & 5\end{smallmatrix}\right) \right\rangle$ & $\left[  \frac{1}{13^2}Q_5(t), \frac{1}{13^3}R_5(t) \right]$ & \href{https://www.lmfdb.org/EllipticCurve/Q/355008ej1/}{\texttt{355008ej1}}\\ 
$[ \texttt{GL2}, \texttt{13B.4.2} ]$ & $\left\langle \left(\begin{smallmatrix}15 & 15 \\ 20 & 5\end{smallmatrix}\right),\left(\begin{smallmatrix}22 & 23 \\ 9 & 24\end{smallmatrix}\right),\left(\begin{smallmatrix}7 & 10 \\ 25 & 5\end{smallmatrix}\right),\left(\begin{smallmatrix}7 & 1 \\ 14 & 1\end{smallmatrix}\right) \right\rangle$ & $[ Q_6(t) , R_6(t) ]$ & \href{https://www.lmfdb.org/EllipticCurve/Q/355008ej2/}{\texttt{355008ej2}}\\ 
$[ \texttt{GL2}, \texttt{13B.4.2} ]$ & $\left\langle \left(\begin{smallmatrix}18 & 7 \\ 3 & 16\end{smallmatrix}\right),\left(\begin{smallmatrix}22 & 23 \\ 9 & 24\end{smallmatrix}\right),\left(\begin{smallmatrix}19 & 14 \\ 23 & 1\end{smallmatrix}\right) \right\rangle$ & $[  13^2Q_6(t),13^3R_6(t) ]$ & \href{https://www.lmfdb.org/EllipticCurve/Q/74529q2/}{\texttt{74529q2}}\\ 
\hline \end{tabular}
\end{center}
\begin{align*}
Q_5(t) &=  3^{3}13^{2}t^{2}(t^{2} - 3t - 1)^{2}(t^{4} - t^{3} + 5t^{2} + t + 1)^{3}(t^{8} - 5t^{7} + 7t^{6} - 5t^{5} + 5t^{3} + 7t^{2} + 5t + 1),\\
R_5(t) &= 23^{3}13^{3}t^{3}(t^{2} - 3t - 1)^{3}(t^{2} + 1)(t^{4} - t^{3} + 5t^{2} + t + 1)^{4}(t^{12} - 8t^{11} + 25t^{10} - 44t^{9} + 40t^{8}\\
&\phantom{= 23} + 18t^{7} - 40t^{6} - 18t^{5} + 40t^{4} + 44t^{3} + 25t^{2} + 8t + 1),\\
Q_6(t) &= 3^{3}t^{2}(t^{2} - 3t - 1)^{2}(t^{4} - t^{3} + 5t^{2} + t + 1)^{3}(t^{8} + 235t^{7} + 1207t^{6} + 955t^{5} + 3840t^{4} - 955t^{3}\\
&\phantom{= 23} + 1207t^{2} - 235t + 1),\\
R_6(t) &= 23^{3}t^{3}(t^{2} - 3t - 1)^{3}(t^{2} + 1)(t^{4} - t^{3} + 5t^{2} + t + 1)^{4}(t^{12} - 512t^{11} - 13079t^{10} - 32300t^{9}\\
&\phantom{= 23} - 104792t^{8} - 111870t^{7} - 419368t^{6} + 111870t^{5} - 104792t^{4} + 32300t^{3} - 13079t^{2}\\
&\phantom{= 23} + 512t + 1).
\end{align*}

\subsection{Genus 1 groups}\label{subsection:genus1Table}
In this table, the parametrization of the modular curve comes as a Weierstrass equation for the genus 1 modular curve and below we give the $j$-map. 
\begin{center}
\renewcommand{\arraystretch}{1.5}
\begin{tabular}{| c | c | c | c |}\hline
\multicolumn{4}{| c |}{ \textbf{$(2,7)$-entanglements of type $\zZ{2}$}}   \\\hline
\textsc{Label} & \textsc{Generators} & \textsc{Model for $X_G$} & \textsc{Example} \\\hline\hline
\texttt{[GL2,7Ns]} &$\left\langle 
\left(\begin{smallmatrix} 3 & 2 \\ 5 & 11 \end{smallmatrix}\right),
\left(\begin{smallmatrix} 2 & 13 \\ 1 & 2 \end{smallmatrix}\right),
\left(\begin{smallmatrix} 0 & 9 \\ 5 & 0 \end{smallmatrix}\right),
\left(\begin{smallmatrix} 13 & 12 \\ 0 & 1 \end{smallmatrix}\right)
 \right\rangle$  &  $y^2 = x^3 - 4x^2 + 3x + 1$  &  \href{https://www.lmfdb.org/EllipticCurve/Q/361a1/}{\texttt{361a1}}\\
\texttt{[GL2,7Nn]} &$\left\langle 
\left(\begin{smallmatrix} 12 & 11 \\ 11 & 3 \end{smallmatrix}\right),
\left(\begin{smallmatrix} 0 & 9 \\ 9 & 13 \end{smallmatrix}\right),
\left(\begin{smallmatrix} 8 & 9 \\ 1 & 6 \end{smallmatrix}\right) 
 \right\rangle$  &  $y^2 = x^3 - 4x^2 + 3x + 1$  &  \href{https://www.lmfdb.org/EllipticCurve/Q/121b1/}{\texttt{121b1}} \\\hline
\end{tabular}
\end{center}
\vspace{7pt}
The $j$-map for the group with label \texttt{[GL2,7Ns]} is given by
\begin{align*}
j_s((x,y)) &= \frac{x(x+1)^3(x^2-5x+1)^3(x^2-5x+8)^3(x^4 - 5x^3 + 8x^2 - 7x + 7)^3}{(x^3 - 4x^2 + 3x + 1)^7},
\end{align*}
while the $j$-map for the group with label \texttt{[GL2,7Nn]} is given by
\begin{align*}
j_n((x,y)) &= \frac{ 8000 }{ 27 }\cdot\frac{x^3(x-2)^3(x-4/3)^3P(x)^3}{Q(x)^7},
\end{align*}
where
\vspace*{-.3em}
\begin{align*}
P(x) &= (x^6 - 20/3x^5 + 148/9x^4 - 181/9x^3 + 134/9x^2 - 56/9x + 14/9)(x^6 - 20/3x^5 + 148/9x^4 -\\
&\phantom{=  x^6 }  160/9x^3 + 64/9x^2 + 7/9)(x^6 - 20/3x^5 + 148/9x^4 - 842/45x^3 + 92/9x^2 - 112/45x - 7/45)\\
Q(x) &= x^9 - 10x^8 + 124/3x^7 - 2503/27x^6 + 1132/9x^5 - 2956/27x^4 + 551/9x^3 - 518/27x^2 + \\ 
&\phantom{=  x^6 } 56/27x + 7/27.
\end{align*}

\bibliography{refs}{}
\bibliographystyle{amsalpha}
\end{document}